\documentclass{amsart}

\usepackage{graphicx}

\usepackage{amsmath,amscd}
\usepackage{enumerate}
\usepackage{mathrsfs}
\usepackage{mathtools}
\usepackage{tikz-cd}
\usepackage{array} 
\usepackage[utf8]{inputenc}
\inputencoding{latin1}

\theoremstyle{definition}

\theoremstyle{remark}

\numberwithin{equation}{section}

\theoremstyle{definition}
\newtheorem{defi}{Definition}[section]
\newtheorem{thm}[defi]{Theorem}
\newtheorem{prop}[defi]{Proposition}
\newtheorem{lem}[defi]{Lemma}
\newtheorem{cor}[defi]{Corollary}
\newtheorem{eg}[defi]{Example}
\newtheorem{rmk}[defi]{Remark}

\newcommand{\ts}{\textsuperscript}
\newcommand{\bb}[1]{\mathbb{#1}}
\newcommand{\mcal}[1]{\mathcal{#1}}
\newcommand{\cohm}[3]{H^{#1}(#2,#3)}
\newcommand{\der}[2]{\dfrac{\partial#1}{\partial#2}}
\newcommand{\hol}{\overline{\partial}}

\begin{document}

\title{Monopoles on $\bb R^5$ and Generalized Nahm's equations}

\author{Rodrigo Pires dos Santos}
\address{Instituto de Matem\'atica, Universidade Estadual de Campinas}
\curraddr{R. Sergio Buarque de Holanda, 651 - Cidade Universitaria, Campinas - SP, Brazil, 13083-859}
\email{rdgpires@gmail.com}


\date{September 26, 2016}



\begin{abstract}
Our approach to define monopoles is twistorial and we start by developing the twistor theory of $\bb R^5$, which is an analogue of the twistor theory for $\bb R^3$ developed by Hitchin. Using this, we describe a Hitchin-Ward transform for $\bb R^5$, that gives monopoles. In order for us to construct monopoles we make use of spectral curves. Then, using those spectral curves we find a new system of equations, analogue to the Nahm's equations.

\end{abstract}

\maketitle

\tableofcontents

\section*{Introduction}

Let $\nabla$ be an $SU(2)$-connection on a complex vector bundle $E$ over $\bb R^3$ and $\phi$ a skew-symmetric section of $End(E)$. A monopole on the Euclidean $\bb R^3$ consists of a pair $(\nabla,\phi)$ minimising, with finite energy, the Yang-Mills-Higgs energy functional
\begin{align*}
	\mcal V=\int_{\bb{R}^3}|F_\nabla|^2+|\nabla\phi|^2.
\end{align*}
One can show \cite{AH} that the pair $(\nabla,\phi)$ is a monopole if and only if it satisfies the Bogomolny equation,
\begin{align*}
	F_\nabla=*\nabla\phi,
\end{align*}
and $\nabla$ and $\phi$ are subject to the boundary conditions:
\begin{align*}
	|\phi|&=1-\frac{m}{2r}+O(r^{-2})\\
	\der{|\phi|}{\Omega}&=O(r^{-2})\\
	|\nabla\phi|&=O(r^{-2}),\,\,\text{as}\,\,r\to\infty,\\
\end{align*}
where $\der{|\phi|}{\Omega}$ is the angular derivative of $|\phi|$. The first boundary condition says that we can restrict $\phi$ to $S_\infty$, the sphere at the infinity, and obtain a map $\frac{\phi}{|\phi|}:S_\infty\to S^2\subset \mathfrak{su}(2)$. By integration on the sphere at the infinity, one can show that the degree of this map is the integer $m$, called the topological charge of the monopole.

Using the geometry of oriented lines in $\bb R^3$, Hitchin proved in \cite{MG} that a solution to the Bogomolny equations correspond to holomorphic bundles on $\bb T$, the total space of $T\bb CP^1$; this type of result is known in the literature as the Hitchin-Ward correspondence. Furthermore, he gave a twistor description of the boundary conditions. Namely, he proved that if a bundle $\tilde E$ on $\bb T$ corresponds to a monopole, then $\tilde E$ is given by an extension of line bundles on $\bb T$. He was also able to determine the bundle $\tilde E$ from an algebraic curve on the twistor space.

Recently, Bielawski \cite{SU2} defined \emph{generalised hypercomplex manifolds} (GHC manifolds for short) that are manifolds whose tangent space at every point decomposes as copies of irreducible representations of $SU(2)$. An important feature of GHC manifolds is that they possess a twistor interpretation similar to Hitchin's twistor description of $\bb R^3$. Thus, we can describe \emph{Bogomolny pairs}, a generalization to the Bogomolny equations (this also generalizes the Bogomolny Hierarchy discussed in \cite{MS}). More specifically, a \emph{Bogomolny pair} on a GHC manifold $M$ is a pair $(\nabla,\Phi)$, where $\nabla$ is a connection on a complex vector bundle $E$ and $\Phi$ a tuple of endomorphisms of $E$, such that $\nabla\oplus\Phi$ is flat over certain subspaces of $M$ called $\alpha$-surfaces. There is a  Hitchin-Ward correspondence for $M$ giving a correspondence between Bogomolny pairs and holomorphic bundles on the twistor space of $M$ that are trivial on real sections.


This paper presents an approach to the construction of monopoles on $\bb R^5$ and is organized as follows: In the first section \ref{sec:GHC} we recover the results of \cite{SU2} on GHC manifolds. Then, we define a GHC structure on $\bb R^5$ by describing it as the space of real sections of the line bundle $\mcal O(4)$ over $\bb CP^1$. The second section \ref{section2} is devoted to the description of the Hitchin-Ward correspondence for $\bb R^5$. We use the proof of the correspondence to find a distinguished line bundle $L$ on the twistor space that corresponds to a trivial Bogomolny pair on $\bb R^5$; this bundle will play an important role in the construction of monopoles. In section \ref{section3} we initiate the construction of monopoles. We begin with a discussion on how a spectral curve can be used to build a pair $(\nabla,\Phi)$ on $\bb R^5$ and prove that spectral curves gives rise to solutions to the generalized Nahm's equations. In section \ref{section4} we deduce the boundary conditions for the generalized Nahm's equations. Namely, we prove an equivalence between the following:
\begin{enumerate}[(1)]
	\item A compact algebraic curve $S$ in  the total space of $\mcal O(4)$ such that:
		\begin{itemize}
		\item $S$ is a compact algebraic curve in the linear system $|\mathcal{O}(4k)|$, 
		\item $S$ has no multiple components,
		\item the line bundle $L$ has order $2$ on $S$ and
		\item $H^0(S,L^z(2k-3))=0$ for $z\in (0,2)$.
		\end{itemize}	
	\item A solution to the system of equations:
	\begin{align*}
	\begin{array}{ll}
		\dot{A}_0&=\dfrac{1}{2}[A_0,A_2]\\
		\dot{A}_1&=[A_0,A_3]+\dfrac{1}{2}[A_1,A_2]\\
		\dot{A}_2&=[A_1,A_3]+[A_0,A_4]\\
		\dot{A}_3&=[A_1,A_4]+\dfrac{1}{2}[A_2,A_3]\\
		\dot{A}_4&=\dfrac{1}{2}[A_2,A_4],
	\end{array}	
	\end{align*}
	where $A_j(s)$ is a $k\times k$ matrix for $z\in (0,2)$ and such that:
		\begin{itemize}
		\item $A_1$ and $A_3$ are analytic on the whole interval $[0,2]$;
		\item $A_0$, $A_4$ and $A_2$ have simple poles at $0$ and $2$, but are otherwise analytic;
		\item The residues of $A_0$, $A_4$ and $A_2$ at $z=0$ and $z=2$ define an irreducible $k$-dimensional representation of $\mathfrak{sl}(2,\bb C)$.
		\end{itemize}
	\end{enumerate}
	
	The formality of the proof follows the idea for the construction of monopoloes on $\bb R^3$ done in \cite{CM} and \cite{HM}. However, it is important to highlight the main differences. First, since we do not define monopoles from an energy functional, we do not have a topological definition of charge as in the $\bb R^3$ case. Thus, we use the degree of the spectral curve as a parameter of solutions. Another difference is the proof of proposition \ref{prop:bextension}, which is essential to the proof of the main result. This proposition is the analogue to proposition 5.13 in \cite{CM}. Hitchin's proof relies on an $SL(2,\bb C)$ invariance of the construction and our proof consists of making a more explicit calculation since we do not have a group invariance.

\subsection*{Acknowledgement}
  I want to thank Roger Bielawski for proposing this topic for my PhD and for his guidance during the research. I also want to thank Marcos Jardim, Derek Harland and Paul Sutcliffe for useful suggestions. This paper derives from the work culminating in my PhD thesis and it was fully funded by CAPES PhD Scholarship number BEX:5705/10-0.

\section{Generalised hypercomplex manifolds}\label{sec:GHC}

The main purpose of this section is to introduce the concepts of generalized hypercomplex manifolds. The main reference is \cite{SU2}. 

\begin{defi}
Let $M$ be a smooth manifold. A \textit{generalised almost complex manifold} is a smooth fibrewise action of $SU(2)$ in the tangent bundle such that each $T_xM$ decomposes as $V\otimes\bb R^n$, where $V$ is an irreducible representation of $SU(2)$. The complexified representation $V^{\bb C}$ is one or two copies of the $k\ts{th}$-symmetric power of the defining representation of $SL(2,\bb C)$. We shall then call $M$ an \textit{almost k-hypercomplex manifold}.
\end{defi}




 
We can produce examples of those structures by looking into the space of sections of holomorphic bundles over $\bb CP^1$. This happens because irreducible representations of $SL(2,\bb C)$ can be realised as sections of line bundles over $\bb CP^1$. A $\sigma$-bundle (or \textit{real} bundle) over $\bb CP^1$ is a holomorphic bundle $E$ equipped with a anti-holomorphic involution $\sigma$ covering the antipodal map on $\bb C^1$, a \textit{real section} of a $\sigma$-bundle is a section invariant under the involution $\sigma$. The map $\sigma$ will be called a \textit{real structure}. Following these definitions, we can describe the irreducible representation of $SU(2)$ as real sections of a $\sigma$-bundle over $\bb CP^1$. Consequently, the tangent space of a generalised almost hypercomplex manifold is the space of real sections of a $\sigma$-bundle. In fact:

\begin{prop}[ \cite{SU2} proposition 2.2] Let $Z$ be a complex manifold fibering over $\bb CP^1$ equipped with an anti-holomorphic involution $\tau$ which covers the antipodal map on $\bb CP^1$. Suppose there exists a holomorphic real section of $Z\to\bb CP^1$ whose normal bundle is isomorphic to $\mcal O(k)\otimes\bb C^n$ ($k>0$), then the space of such real sections is an almost $k$-hypercomplex manifold of dimension $n(k+1)$.
\end{prop}

This proposition motivates the following definition:

\begin{defi}
An almost $k$-hypercomplex structure on a manifold $M$ is \textit{integrable} if $M$, together with the $SU(2)$ action on its tangent bundle, can be described (locally) as the space of real sections of a complex manifold $Z$ fibering over $\bb CP^1$. We shall say that $M$ is a \textit{generalised hypercomplex manifold} or GHC manifold for short. The space $Z$ is called the \textit{twistor space} of $M$.
\end{defi}



\begin{eg}\label{eg:R^k+1}






Let $H$ be the $k$-dimensional, for $k$ even, irreducible representation of $SL(2,\mathbb{C})$, then it acts irreducibly on the dual $H^*$. Let $B$ be a Borel subgroup of $SL(2,\mathbb{C})$, then $SL(2,\mathbb{C})/B\cong\mathbb{C}P^1$.  For each $q\in\mathbb{C}P^1$, let $B_q$ be its corresponding Borel subgroup and $l_q$ be the highest weight vectors for $B_q$. This gives an injective map $\mathbb{C}P^1\mapsto\mathbb{P}(H^*)$ and let $\tilde{L}_k$ be the bundle on $\mathbb{C}P^1$ given by the pullback of the tautological bundle on $\mathbb{P}(H^*)$. For $L_k=(\tilde{L}_k)^*$ we have:


	\begin{thm}\label{thm:Borel-Weil}
		(Borel-Weil theorem) In the notation of the example above,
		\begin{align*}
			H^0(G/B,L_k)\cong H.
		\end{align*}
	\end{thm}
Since $k$ is even, we can endow $H$ with a real structure and then $H^\bb R$ is a GHC-manifold with twistor space $L_k$. We shall later describe explicitly the $\alpha$-surfaces when $k=4$.
	






\end{eg}


Let $M$ be a GHC manifold and consider the action of $SL(2,\bb C)$ on the complexified cotangent bundle $T^*M^{\bb C}$. For each point $q\in\bb CP^1$ let $B_q$ be the corresponding Borel subgroup of $SL(2,\bb C)$. Define the following:
	\begin{enumerate}[i)]
	\item $\mcal U_q$ is the subbundle of $(T^*M)$ corresponding to the highest weight with respect to $B_q$,
	\item $\mcal K_q$ is the subbundle of $TM^{\bb C}$ annihilated by $\mcal U_q$ and
	\item  $\mcal F_q=\mcal K_q\cap\overline{\mcal K_q}\cap TM$ is a distribution on $M$.
	\end{enumerate} 
We then have:

\begin{thm}\label{thm:integrability}
(\cite{SU2} theorem 2.5) An almost $k$-hypercomplex structure on a manifold $M$ is integrable if and only if for every $q\in\bb CP^1$ the subbundle $\mcal K_q$ is involutive for all $q\in\bb CP^1$, this is to say, $[\mcal K_q,\mcal K_q]\subset\mcal K_q$.
\end{thm}

We shall not prove the theorem above, however we shall see how it can be used to construct the twistor space of a GHC-manifold.

Define the \textit{twistor distribution} $\mcal Z$ of $M$ to be the distribution on $M\times \bb CP^1$ given by $\mcal Z_{(m,q)}=((\mcal F_q)_m,0)$. The theorem above says that this distribution is involutive and thus it defines a foliation of $M\times\bb CP^1$. Moreover, the leaf space $Z=(M\times\bb CP^1)/\mcal Z$ is the twistor space of the GHC-manifold $M$. If the foliation is simple, then $Z$ is a complex manifold and the projection $\eta:M\times\bb CP^1\to Z$ is a surjective submersion, in this case $M$ is called a \textit{regular} GHC-manifold. The leaves of the foliation $\mcal Z$ will be called $\alpha$-surfaces.

Let $M$ be a GHC-manifold, then it is given as the space of real sections of a fibration $Z\to \bb CP^1$, then $M$ has a natural complexification $M^{\bb C}$, it is the space of all sections of the fibration.  Notice that the holomorphic tangent bundle $TM^{(1,0)}$ of $M^{\bb C}$ is then endowed with a holomorphic action of $SL(2,\bb C)$ such that $TM^{(1,0)}=S^k\bb C^2\otimes\bb C^n$. 


We now start the description of a  the twistor theory of a GHC manifold. We start by describing some distinguished bundles on a GHC-manifold M. But first we need some results regarding bundles on $\bb CP^1$ and representations of $SL(2,\bb C)$. In the remaining of this section we denote $G=SL(2,\bb C)$ and $B$ is the Borel subgroup of the upper diagonal matrices.

Let $L=\mcal O(k)$ be the degree $k$ line bundle on $\bb CP^1$, for $k>0$. Then the space of sections $H$ is an irreducible representation of $G$ from \eqref{thm:Borel-Weil}. Notice that the homogeneous bundle $\underline{H}=G\times_B H$ is trivial and that we have an equivariant map:
$$ \underline{H}\to L,$$
which is given by evaluation. Namely, if $h\in H$ and $q\in G/B\cong\bb CP^1$ the map above sends $(h,q)$ to $h(q)$.

Now, define a bundle $K$ on $\bb CP^1$ given by the exact sequence of homogeneous bundles:
	\begin{align}\label{eq:sequence1}
		0\to K\to \underline{H}\to L\to0.
	\end{align}
The cohomology exact sequence of the dual to the sequence \eqref{eq:sequence1} gives an exact sequence of $G$ representations:
	\begin{align}\label{eq:sequence2}
		0\to H^*\xrightarrow{i} \hat{H}\xrightarrow{j} H'\to0,
	\end{align}
where $H'=H^1(L^*)\cong S^{k-2}\bb C^2$, $\hat{H}=H^0(K^*)$ and notice that $H^*\cong H^0(\underline{H^*})$, since $H$ is a trivial bundle. Moreover, the sequence \eqref{eq:sequence2} is split.




Now for each $q\in\bb CP^1$ we notice that the line of highest weight vectors in $H$, denoted by $S_q$, are contained in $K$. Therefore, we can define a subbundle $S$ of $K$ whose fibre at $q$ is $S_q$.

We then consider the short exact sequence:
	\begin{align}\label{sequence3}
		0\to (K/S)^*\rightarrow K^*\rightarrow S^*\to0.
	\end{align}
Its long exact sequence in cohomology starts as:
	\begin{align}\label{sequence4}
		0\to H^0((K/S)^*)\rightarrow \hat{H}\rightarrow H^0(S^*).
	\end{align}
Now Borel-Weil theorem says that $H^*$ and $H^0(S^*)$ are isomorphic representations of $G$. Thus, we obtain a map $p:\hat{H}\to H^*$. It is proved in \cite{SU2} lemma 3.3 that $p$ is the left inverse for the map $i$ in \eqref{eq:sequence2}.

We can now state and prove:
	\begin{prop}\label{prop:emap}
	We have an isomorphism of homogeneous bundles
		$$ (K/S)^*\cong G\times_BH'.$$
	In particular, $H^0((K/S)^*)\cong H^1(\bb CP^1,L^*)$ and $K/S$ is trivial.
	\end{prop}
	\begin{proof}
		$H$ is isomorphic to $S^k\bb C^2$ as a representation of $G$ and we shall write the vectors of $H$ as $(v_0,v_1,\cdots,v_k)$ where the coordinates are relative to the weight decomposition with respect to $B$, where $v_0$ correspond to the minimal weight and $v_k$, the maximal weight. The fibre $K_{[1]}$ of $K$ at the point $[1]\in G/B\cong\bb CP^1$ is given by vectors of the form $(0,v_1,\cdots,v_k)$ and the fibre $S_{[1]}$, by $(0,\cdots,0,v_k)$. The map $K_{[1]}/S_{[1]}\to S^{k-2}\bb C^2$ induced by
		$$ (0,v_1,\cdots,v_k)\mapsto (v_1,\cdots,v_{k-1})$$
		is an isomorphism of $B$-modules. Since the bundles are homogeneous we have an isomorphism of bundles.
	\end{proof}
	

We can now return our attentions to differential geometry. Let $M$ be a regular GHC-manifold and $Z$ its twistor space. Therefore, on the complexified case, we have the double fibration:
	\begin{align}\label{dublefibration}
		Z\xleftarrow{\eta}Y=M^{\bb C}\times \bb CP^1\xrightarrow{p}M^{\bb C}.
	\end{align}
	\begin{defi}
		The sheaf of $\eta$-\textit{vertical holomorphic $l$-forms} $\Omega^l_\eta$ is defined by
		$$ \Omega^l_\eta=\Lambda^l(\Omega^1(Y)/\eta^*(\Omega^1(Z))).$$
	\end{defi}
	
	\begin{prop}
		We have an isomorphism of sheaves $p_*(\Omega^1_\eta)\cong E^*\otimes\hat{H}$, where $\hat{H}$ is defined in the sequence \eqref{eq:sequence2}.
	\end{prop}
	\begin{proof}
	Let $x\in M^\bb C$ and let $\bb CP^1_x$ be the fibre of $p$ over $x$. The $\eta$-normal bundle of $\bb CP^1_x$ in $Y$, this is to say, the normal bundle of $\bb CP^1_x$ along the fibres of $\eta$, is the bundle whose fibre at $(x,q)\in \bb CP^1_x$ is $\mcal K_q$. From the definition of push-forward we have:
	$$p_*(\Omega^1_\eta)=H^0(\bb CP^1_x,\mcal K^*). $$
	Now we have the decompositions $TM^\bb C=E_M\otimes H$ and $\mcal K=\bb C^n\otimes K$, where $K$ is defined in \eqref{eq:sequence1}. Since $H^0(\bb CP^1,K^*)=\hat H$, we have proved the proposition.
	\end{proof}
	
	We now state a result that will be necessary later.
	
	\begin{prop}\label{2forms}
		We have a splittings:
\begin{itemize}
		\item $p_*(\Omega^1_\eta)\cong\Omega^1(M^{\bb C})\oplus(E^*\otimes H')$.
		\item $p_*(\Omega^2_\eta)\cong(S^2E^*\otimes H_-)\oplus(\Lambda^2E^*\otimes H_+),$ where $$H_-=H^0(\bb CP^1,\Lambda^2K^*)\,\, \text{and}\,\, H_+=H^0(\bb CP^1,S^2K^*).$$
\end{itemize}
	\end{prop}
	
	\subsection{$\bb R^5$ as a GHC-manifold and its twistor theory}\label{subsec:r5asghc}


Example \eqref{eg:R^k+1} defines $\bb R^5$ as a $4$-GHC manifold. In this section we shall describe explicitly the twistor distribution for $\bb R^5$.

First we shall fix some notations that will be used throughout this paper. Let $\bb CP^1=\bb C\cup \{ \infty \}$ and put coordinates $\xi$ on $U=\bb C\subset \bb CP^1$ and $\xi'$ on $U'=(\bb C\setminus \{0\})\cup\{\infty\}$ such that $\xi'=\dfrac{1}{\xi}$ on $U\cap U'$.

 We can now fix holomorphic coordinates on $\mathcal{O}(k)$. Let $\pi:\mathcal{O}(k)\to\mathbb{C}P^1$ be the projection and define the open sets $U_0=\pi^{-1}(U)$ and $U_1=\pi^{-1}(U')$. Put coordinates  $(\eta,\xi)$ in $U_0$ and $(\eta',\xi')$ on $U_1$ such that $\eta'=\eta/\xi^k$. Furthermore, from now on, whenever we refer to the total space of the bundle $\mathcal{O}(k)$, we shall name it $\bb T$.
 
 Under these coordinates we can express a holomorphic section $p$ of $\mcal O(k)$ as a polynomial of degree $k$ in $\xi$, namely $p(\xi)=a_0+a_1\xi+\cdots+a_k\xi^k$.  We can define an anti-holomorphic involution in the total space of $\mathcal{O}(k)$, in local coordinates, by $\tau(\eta,\xi)=(\overline{\eta}/\overline{\xi}^k,-1/\overline{\xi})$. Observe that $\tau$ covers the antipodal map in $\mathbb{C}P^1$ and therefore swaps the open sets $U_0$ and $U_1$. This map induces an involution in $H^0(\mathbb{C}P^1,\mathcal{O}(k))$, which will still be called by $\tau$, in the following way: If $p(\xi)=a_0+a_1\xi+\cdots+a_k\xi^k$ is a holomorphic section of $\mathcal{O}(k)$, then $\tau(p)=b_0+\cdots+b_k\xi^k$, where $b_j=(-1)^j\overline{a}_{k-j}$. For a point $(\eta,\lambda)\in\mcal O(k)$ we can define the $\alpha$-\textit{surface} $\Pi_{(\eta,\lambda)}=\{p(\xi)\in\bb C^5|\,p(\lambda)=\eta\}$.
 

We can now concentrate on $\bb R^5$. A point $(x_0,x_1,x_2,x_3,x_4)\in\mathbb{R}^5$ corresponds to the section $p(\xi)=(x_0+ix_4)+(x_1+ix_3)\xi+x_2\xi^2-(x_1-ix_3)\xi^3+(x_0-ix_4)\xi^4\in H^0(\mathbb{C}P^1,\mathcal{O}(4))$. Conversely, given a point $z\in Z$, we define the \textit{real $\alpha$-surface} corresponding to $z$, denoted by $P_z$, to be the subspace in $\mathbb{R}^5$ consisting of real sections through $z$. Namely, we can consider $z\in U_0$ so that we can write $z=(\eta_0,\xi_0)$ in local coordinates, then we have $P_z=\{p\in H^0(\mathbb{C}P^1,\mathcal{O}(4))|\,\,p(\xi_0)=\eta_0\}$.


We define $\bb C^5$ as the fourth symmetric power of the defining representation of $SL(2,\bb C)$, therefore it can be described as the space of polynomials of degree $4$ in $\xi$ and the explicit action of $SL(2,\bb C)$ on $\bb C^5$ is given by:
	\begin{align}\label{eq:action1}
		g\cdot p(\xi)=(c\xi+d)^4 \cdot p\left( \frac{a\xi+b}{c\xi+d}\right),
	\end{align}
where $g=\left(\begin{smallmatrix} a &  b\\  c & d \end{smallmatrix} \right)\in SL(2,\bb C)$ and $p(\xi)\in\bb C^5$. We can understand this action as being induced by the action of $SL(2,\bb C)$ in the total space of $\mcal O(4)$ defined by
	\begin{align}\label{eq:action2}
		g\cdot (\eta,\xi)=\left( \frac{\eta}{(c\xi+d)^4},\frac{a\xi+b}{c\xi+d}\right).
	\end{align}
	
	For the following proposition, we write an element $g\in SU(2)\subset SL(2,\bb C)$ as $ g=\left(\begin{smallmatrix} \alpha &-\overline{\beta}\\ \beta   &\overline{\alpha} \end{smallmatrix} \right)$.




	 
	\begin{prop}
		This action is compatible with the real structure $\tau$ in $\mathcal{O}(4)$, this is to say, $\tau g= g \tau$ for all $g\in SU(2)$.
	\end{prop}
	\begin{proof}
		The proof follows by direct computation using the action \eqref{eq:action2} and the definition of $\tau$. We have:
			\begin{align*}
				g\cdot\tau (\eta,\xi)=\left(\frac{\overline{\eta}}{(\overline{\alpha}\overline{\xi}+\overline{\beta})^4}, \frac{\overline{\beta}\overline{\xi}-\overline{\alpha}}{(\overline{\alpha}\overline{\xi}+\overline{\beta})}\right) =\tau g\cdot(\eta,\xi).
			\end{align*}
	\end{proof}
	
	For a point $\lambda\in U\subset\mathbb{C}P^1$ define $g_\lambda\in SU(2)$ by $g_\lambda=\frac{1}{\sqrt{1+\overline{\lambda}\lambda}}\begin{pmatrix} 1&\lambda \\  -\overline{\lambda} & 1  \end{pmatrix}$ and notice that $g_\lambda$ is the unique, up to a $U(1)$ multiplication, element in $SU(2)$ such that $g^{-1}_\lambda\cdot (0,0)=(0,\lambda)$.

	Now we shall explicitly describe the bundle $K$, which is defined in \eqref{eq:sequence1}. First, we identify the tangent space $T_x\bb C^5$ at $x\in \bb C^5$ with $S^4\bb C^2$ and denote a vector in $T_x\bb C^5$ as a polynomial of degree $4$ in $\xi$. We then have:
	
	\begin{prop}\label{twistorfibration1}
		Let $\lambda\in U\subset\bb CP^1$, then the fibre $K_\lambda=Span_\bb C\{V_1^\lambda,V_2^\lambda,V_3^\lambda,V_4^\lambda\}$, where
			\begin{itemize}
			\item $V_1^\lambda=g_\lambda^{-1}\cdot\xi=\dfrac{1}{(1+\lambda\overline{\lambda})^2}(\overline{\lambda}\xi+1)^3(\xi-\lambda)$;
			\item $V_2^\lambda=g_\lambda^{-1}\cdot\xi^2=\dfrac{1}{(1+\lambda\overline{\lambda})^2}(\overline{\lambda}\xi+1)^2(\xi-\lambda)^2$;
			\item $V_3^\lambda=g_\lambda^{-1}\cdot\xi^3=\dfrac{1}{(1+\lambda\overline{\lambda})^2}(\overline{\lambda}\xi+1)(\xi-\lambda)^3$;
			\item $V_4^\lambda=g_\lambda^{-1}\cdot\xi^4=\dfrac{1}{(1+\lambda\overline{\lambda})^2}(\xi-\lambda)^4$.
			\end{itemize}
	\end{prop}
	\begin{proof}
		First remember that the fibre $K_\lambda$ is given by the holomorphic sections $p$ of $\mcal O(4)$ such that $p(\lambda)=0$. Then, notice that $K_0=Span_\bb C\{\xi,\xi^2,\xi^3,\xi^4\}$. Since the group action is an endomorphism, the subspace of $H^0(\bb CP^1,\mcal O(4))$ generated by the $V^\lambda_k$s is a basis for $K_\lambda$. This proves the proposition.
	\end{proof}
	
	\begin{rmk}
		It is important to highlight the use of the group action in the proof above. It will be important when we discuss aspects of the twistor theory of $\bb R^5$ that are invariant under the group action.
	\end{rmk}
	

	We have that in our case $\mcal K_q= K_q$, for all $q\in\bb CP^1$. Therefore, applying the reality condition we have:
	
	\begin{prop}\label{twistorfibration2}
		The twistor distribution $\mathscr F$ on $\bb R^5\times\bb CP^1$ is given by  $\mathscr{F}_{(x,\lambda)}=Span_\mathbb{R}\{(v_1^\lambda,0),(v_2^\lambda,0),(v_3^\lambda,0)\}$, where 
		\begin{align} \label{eq:generators}
				v_1^\lambda&=g_\lambda^{-1}\cdot(\xi-\xi^3)=\frac{(1+\overline{\lambda}\xi)^3(\xi-\lambda)-(1+\overline{\lambda}\xi)(\xi-\lambda)^3}{(1+\lambda\overline{\lambda})^2},\\
				 v_2^\lambda&=g_\lambda^{-1}\cdot\xi^2=\frac{(1+\overline{\lambda}\xi)^2(\xi-\lambda)^2}{(1+\lambda\overline{\lambda})^2}\,\,\text{and}\\
			 	 v_3^\lambda&=g_\lambda^{-1}\cdot i(\xi+\xi^3)=i\left(\frac{(1+\overline{\lambda}\xi)^3(\xi-\lambda)+(1+\overline{\lambda}\xi)(\xi-\lambda)^3}{(1+\lambda\overline{\lambda})^2}\right).
		\end{align}
		Moreover, fixing an ordered frame $\{(v_1^\lambda,0),(v_2^\lambda,0),(v_3^\lambda,0)\}$ for the twistor distribution gives an orientation for the vector space $\mathscr F_{(x,\lambda)}$.
	\end{prop}
	
	\begin{rmk}
		We are describing $\bb R^5$ as the real form of the fourth symmetric power of the defining representation of $SU(2)$. Let $B_q$ be the Borel subgroup of $SU(2)$ corresponding to $q\in\bb CP^1$. If we consider the weight decomposition of $\bb R^5$ with respect to $B_q$, we must have that $v_1$, $v_2$ and $v_3$ are the weight-vectors corresponding to the weights $-2$, $0$ and $+2$ respectively. Therefore, the orientation mentioned in the proposition above is natural with respect to the $SU(2)$ action.
	\end{rmk}

	
\subsection{Invariant metric on $\bb R^5$, $\alpha$-surfaces and further properties}
	
	
	Since we shall need to identify $T\bb R^5$ and $T^*\bb R^5$, we need an $SU(2)$-invariant metric for this:
	
	\begin{prop}\label{invariantmetric}
		(\cite{Mukai} page 27 proposition 1.25) Let $p(\xi)=a_0+a_1\xi+a_2\xi^2+a_3\xi^3+a_4\xi^4$ as a point in $T_x\bb C^5$. Define the quadratic form on $T_x\bb C^5$ by  $N(p)=a_2^2-3a_1a_3+12a_0a_4$. Then, $N$ is $SL(2,\bb C)$-invariant, this is to say, $N(g\cdot p)=N(p)$, for all $p\in T_x\bb C^5$ and $g\in SL(2,\bb C)$.
	\end{prop}

	We can apply the reality condition and restrict this form to the tangent space $T_x\bb R^5$ for $x\in\bb R^5$. For a tangent vector $p(\xi)=(x_0+ix_4)+(x_1+ix_3)\xi+x_2\xi^2-(x_1-ix_3)\xi^3+(x_0-ix_4)\xi^4\in T_x\bb R^5$, we have 
	$$N(p)=x_2^2+3(x_1^2+x_3^2)+12(x_0^2+x_4^2). $$
	
	Thus, $N(p)$ is positive definite and defines an $SU(2)$-invariant metric $g$ on $\bb R^5$ by the polarisation formula. Moreover, we must have that $\{v_1^\lambda,v_2^\lambda,v_3^\lambda\}$, defined in proposition \eqref{twistorfibration2}, is an orthogonal frame for the twistor distribution $\mathscr F$.


	We now turn to the description of the leaves of the twistor foliation, the so called $\alpha$-surfaces. Let $z\in\bb T$, we define $\Pi_z$ to be the space of section of $\mcal O(4)$ that contains $z$, in local coordinates, $\Pi_{(\eta,\xi)}=\{p\in\mcal O(4)| \,p(\xi)=\eta\}$. Applying the reality structure, we define $P_z=\Pi_z\cap\tau(\Pi_z)\cap\bb R^5$. The following proposition follows from propositions \ref{twistorfibration1} and \ref{twistorfibration2}.

	\begin{prop}\label{prop:twistorplanes}
		Let $(\eta,\lambda)\in U_0$, then
		\begin{multline*}
		\Pi_{(\eta,\lambda)}=\left\{ \frac{1}{(1+\lambda\overline{\lambda})^4}\left[\eta(1+\overline{\lambda}\xi)^4+ a_1(1+\overline{\lambda}\xi)^3(\xi-\lambda)+a_2(1+\overline{\lambda}\xi)^2(\xi-\lambda)^2+\right. \right. \\
		\left.\left. +a_3(1+\overline{\lambda}\xi)(\xi-\lambda)^3+a_4(\xi-\lambda)^4 \right] | a_1,\,a_2,\,a_3,\,a_4\in\bb C\right\}.
		\end{multline*}
		Applying the reality condition:
		\begin{multline*}
			P_{(\eta,\lambda)}=\left\{ \frac{1}{(1+\lambda\overline{\lambda})^4}\left( \eta(1+\overline{\lambda}\xi)^4+\overline{\eta}(\xi-\lambda)^4+\right.\right. \\ \left.\left.x_1[(1+\overline{\lambda}\xi)^3(\xi-\lambda)-(1+\overline{\lambda}\xi)(\xi-\lambda)^3]+\right.\right. \\ \left.\left. x_2(1+\overline{\lambda}\xi)^2(\xi-\lambda)^2-x_3[(1+\overline{\lambda}\xi)^3(\xi-\lambda)+(1+\overline{\lambda}\xi)(\xi-\lambda)^3]\right)|x_1,\,x_2,\,x_3\in\bb R  \right\}.
		\end{multline*}
	\end{prop}
					

	Now we shall concentrate on the tangent space to the $\alpha$-surfaces. We shall use the isomorphism $T\mathbb{R}^5\cong T^*\mathbb{R}^5$ given by the above inner product and define what we shall call ``natural forms'' on $\Omega^{0,1}(\mathcal{O}(4))$.

The tangent space of the $\alpha$-surface $P_{(\eta,\lambda)}$, $\lambda\in\mathbb{C}P^1$, is generated by vectors $v_1^\lambda,v_2^\lambda,v_3^\lambda$, where
	\begin{align}\label{eq:basistwistor}
	v_1^\lambda&=\frac{1}{(1+\lambda\overline{\lambda})^2}[(1+\overline{\lambda}\xi)^3(\xi-\lambda)-(1+\overline{\lambda}\xi)(\xi-\lambda)^3],\\
	v_2^\lambda&=\frac{1}{(1+\lambda\overline{\lambda})^2}[(1+\overline{\lambda}\xi)^2(\xi-\lambda)^2],\\
	v_3^\lambda&=\frac{1}{(1+\lambda\overline{\lambda})^2}[(1+\overline{\lambda}\xi)^3(\xi-\lambda)+(1+\overline{\lambda}\xi)(\xi-\lambda)^3],
	\end{align}
where $\lambda$ is the holomorphic coordinate for a point in $U_0=\mathbb{C}P^1\setminus \{\infty\}$.

Using the metric, we can find the dual to the basis above. Namely, we define $\omega^\lambda_j\coloneqq g(v_j^\lambda,\cdot)\in\Omega^1P_{(\eta,\lambda)}$. Using holomorphic coordinates $(a_0, a_1, a_2, a_3, a_4)$ for $\bb C^5$ we can write a frame for $(1,0)$-forms as $\{da_0, da_1, da_2, da_3, da_4\}$. Expanding the formulas for $v_j^\lambda$ above we get:

	\begin{align*}
	\left.
	\begin{array}{ll}
	\omega_1^\lambda&=\frac{1}{3(1+\lambda\overline{\lambda})^2}[6\overline{f_0}da_0+\frac{3}{2}\overline{f_1}da_1+f_2da_2+\frac{3}{2}\overline{f_3}da_3+6\overline{f_4}da_4],\\
	\omega_2^\lambda&=\frac{1}{(1+\lambda\overline{\lambda})^2}[6\overline{\lambda}^2da_0^2-3\overline{\lambda}(1-\lambda\overline{\lambda})da_1+\\&(1-4\lambda\overline{\lambda}+(\lambda\overline{\lambda})^2)da_2+3\lambda(1-\lambda\overline{\lambda})da_3+6\lambda^2da_4],\\
	\omega_3^\lambda&=\frac{i}{3(1+\lambda\overline{\lambda})^2}[6\overline{g_0}da_0+\frac{3}{2}\overline{g_1}da_1+\overline{g_2}da_2+\frac{3}{2}\overline{g_3}da_3+6\overline{g_4}da_4],
	\end{array}
	\right.
	\end{align*}
where:
	\begin{align*}
	f_0&=(\lambda^3-\lambda)=\overline{f_4},\\
	f_1&=(-3\overline{\lambda}\lambda+1-3\lambda^2+\overline{\lambda}\lambda^3)=-\overline{f_3},\\
	f_2&=(-3\overline{\lambda}^2\lambda+3\overline{\lambda}-3\overline{\lambda}\lambda^2+3\lambda)=\overline{f_2},\\
	g_0&=-(\lambda^3+\lambda)=-\overline{g_4},\\
	g_1&=(-3\overline{\lambda}\lambda+1+3\lambda^2-\overline{\lambda}\lambda^3)=\overline{g_3},\\
	g_2&=(-3\overline{\lambda}^2\lambda+3\overline{\lambda}+3\overline{\lambda}\lambda^2-3\lambda)=-\overline{g_2}.
	\end{align*}
	
Observe that $(0,\omega_k^\lambda)$ defines a 1-form on $\mathbb{C}P^1\times\mathbb{R}^5$. However, it will be denoted by the same symbol, $\omega_k^\lambda$.

Now we consider a section $s$ of $\eta:\mathbb{C}P^1\times\mathbb{R}^5\to\mathcal{O}(4)$, $\eta(q,m)=m(q)$, and shall find the pull back $\theta_k\coloneqq s^*\omega_k$, notice that $\theta_k$ is independent of the section $s$. In the next section, we shall use $\theta_k^{0,1}$ to describe distinguished bundles on the total space of $\mcal O(4)$ that correspond with a trivial $U(1)$ monopole data. Thus, this method allows us to define line bundles over $\mathcal{O}(4)$ with vanishing first Chern class.

We can choose an explicit section s of $\eta$:

	\begin{align}\label{naturalsection}
	\begin{array}{ll}
	s:\mathcal{O}(4)&\to\mathbb{C}P^1\times\mathbb{R}^5,\\
	(\mu,\lambda)&\mapsto\left(\lambda,\frac{1}{(1+\lambda\overline{\lambda})^2}(xv_0^\lambda+yv_4^\lambda)\right),
	\end{array}
	\end{align}
	where $\mu=x+iy$ and $$v_0^\lambda=\dfrac{1}{(1+\lambda\overline{\lambda})^2}[(1+\overline{\lambda}\xi)^4+(\xi-\lambda)^4] $$ and $$v_4^\lambda=\dfrac{i}{(1+\lambda\overline{\lambda})^2}[(1+\overline{\lambda}\xi)^4-(\xi-\lambda)^4] .$$ The vector fields $v_0^\lambda$ and $v_4^\lambda$ on $\bb R^5$ correspond respectively to the maximal and minimal weights of $\bb R^5$, as a $SU(2)$ representation, with respect to the Borel subgroup $B_\lambda$.

We can now state the following:
					
	\begin{prop}\label{naturalforms}
	The $(0,1)$ parts of the \textit{natural forms} are given by:
		\begin{align*}
			\theta_1^{0,1}&=\frac{3\mu}{(1+\lambda\overline{\lambda})^3}d\overline{\lambda},\\
			\theta_2^{0,1}&=0,\\
			\theta_3^{0,1}&=\frac{3i\mu}{(1+\lambda\overline{\lambda})^3}d\overline{\lambda}.
		\end{align*}
	\end{prop}
	\begin{rmk}
	Before we proceed with the proof of this result, we shall point out that the differential forms above shall be used in the description of distinguished line bundles on the total space of $\bb T$.
	\end{rmk}
	\begin{proof}
		First write $\omega_j^\lambda=\sum_{j=0}^4h^j_k(\lambda)da_k$, where the $h^j_k$s are given by the equations defining $\omega_j$s. The pullback by $s$ is given by: $$s^*\omega_j =\sum_{j=0}^4h^j_k(s(\mu,\lambda))d(a_k(s(\mu,\lambda))),$$
		where $a_k(s(\mu,\lambda))$ is the coordinate function and notice that $h^j_k(s(\mu,\lambda))=h^j_k(\lambda)$. 
		
		Expanding $v_0^\lambda$ and $v_4^\lambda$ above we get
		$$v_0^\lambda=\dfrac{1}{(1+\lambda\overline{\lambda})^2}\left[ (1+\lambda^4)+4(\overline{\lambda}-\lambda^3)\xi+6(\lambda^2+\overline{\lambda}^2)\xi^2-4(\lambda-\overline{\lambda}^3)\xi^3+(1+\overline{\lambda}^4)\xi^4\right]$$
		and
		$$v_4^\lambda=\dfrac{i}{(1+\lambda\overline{\lambda})^2}\left[ (1-\lambda^4)+4(\overline{\lambda}+\lambda^3)\xi+6(\overline{\lambda}^2-\lambda^2)\xi^2+4(\lambda+\overline{\lambda}^3)\xi^3+(\overline{\lambda}^4-1)\xi^4\right].$$
		
		
		From the definition of $s$ we have:
		
			\begin{itemize}
				\item $x_0(s(\mu,\lambda))=\dfrac{1}{(1+\lambda\overline{\lambda})^4}[x(1+\lambda^4)+iy(1-\lambda^4)]=\dfrac{1}{(1+\lambda\overline{\lambda})^4}[\mu+\overline{\mu}\lambda^4],$
				\item $x_1(s(\mu,\lambda))=\dfrac{4}{(1+\lambda\overline{\lambda})^4}[x(\overline{\lambda}-\lambda^3)]+iy(\overline{\lambda}+\lambda^3)=\dfrac{4}{(1+\lambda\overline{\lambda})^4}[\mu\overline{\lambda}+\overline{\mu\lambda^3}],$
				\item $x_2(s(\mu,\lambda))=\dfrac{6}{(1+\lambda\overline{\lambda})^4}[x(\lambda^2+\overline{\lambda}^2)+iy(\overline{\lambda}^2)-\lambda^2]=\dfrac{6}{(1+\lambda\overline{\lambda})^4}[\overline{\mu}\lambda^2+\mu\overline{\lambda}^2],$
				\item $x_3(s(\mu,\lambda))=\dfrac{4}{(1+\lambda\overline{\lambda})^4}[-x(\lambda-\overline{\lambda}^3)+iy(\overline{\lambda}+\lambda)]=\dfrac{4}{(1+\lambda\overline{\lambda})^4}[-\overline{\mu}\lambda+\mu\overline{\lambda}^3]\,\,\text{and}$
				\item $x_4(s(\mu,\lambda))=\dfrac{1}{(1+\lambda\overline{\lambda})^4}[x(1+\lambda^4)+iy(\overline{\lambda}^4-1)]=\dfrac{1}{(1+\lambda\overline{\lambda})^4}[\overline{\mu}+\mu\overline{\lambda}^4].$
			\end{itemize}
			
		Since we are interested only in the $(0,1)$ part of the $s^*\omega_j$s, we shall compute: $$ dx_j(s(\mu,\lambda))^{0,1}=\der{x_j(s(\mu,\lambda))}{\overline{\lambda}}d\overline{\lambda}+\der{x_j(s(\mu,\lambda))}{\overline{\mu}}d\overline{\mu}.$$
		
		Computing the derivatives:
		
		\begin{center}
			 \begin{align*}
				\left\{
                			\begin{array}{ll}
                  				\der{x_0(s(\mu,\lambda))}{\overline{\lambda}}=\dfrac{-4(\mu\lambda+\overline{\mu}\lambda^5)}{(1+\lambda\overline{\lambda})^5},\\
					\der{x_0(s(\mu,\lambda))}{\overline{\mu}}=\dfrac{\lambda^4}{(1+\lambda\overline{\lambda})^4}.
                			\end{array}
              			\right.
			\end{align*}
			
			\begin{align*}
				\left\{
                			\begin{array}{ll}
                  				\der{x_1(s(\mu,\lambda))}{\overline{\lambda}}=\dfrac{4[\mu(1+\lambda\overline{\lambda})-4\lambda(\mu\overline{\lambda}-\overline{\mu}\lambda^3)]}{(1+\lambda\overline{\lambda})^5},\\
					\der{x_1(s(\mu,\lambda))}{\overline{\mu}}=\dfrac{-4\lambda^3}{(1+\lambda\overline{\lambda})^4}.
                			\end{array}
              			\right.
			\end{align*}
			
			\begin{align*}
				\left\{
                			\begin{array}{ll}
                  				\der{x_2(s(\mu,\lambda))}{\overline{\lambda}}=\dfrac{12[\mu\overline{\lambda}(1+\lambda\overline{\lambda})-4\lambda(\overline{\mu}\lambda^2+\mu\overline{\lambda}^2)]}{(1+\lambda\overline{\lambda})^5},\\
					\der{x_2(s(\mu,\lambda))}{\overline{\mu}}=\dfrac{6\lambda^2}{(1+\lambda\overline{\lambda})^4}.
                			\end{array}
              			\right.
			\end{align*}
			
			\begin{align*}
				\left\{
                			\begin{array}{ll}
                  				\der{x_3(s(\mu,\lambda))}{\overline{\lambda}}=\dfrac{4[3\overline{\lambda}^2\mu(1+\lambda\overline{\lambda})-4\lambda(-\overline{\mu}\lambda+\mu\overline{\lambda^3})]}{(1+\lambda\overline{\lambda})^5},\\
					\der{x_3(s(\mu,\lambda))}{\overline{\mu}}=\dfrac{-4\lambda}{(1+\lambda\overline{\lambda})^4}.
                			\end{array}
              			\right.
			\end{align*}
			
			\begin{align*}
				\left\{
                			\begin{array}{ll}
                  				\der{x_4(s(\mu,\lambda))}{\overline{\lambda}}=\dfrac{4\mu\overline{\lambda}^3[(1+\lambda\overline{\lambda})-4\lambda(\overline{\mu}+\mu\overline{\lambda}^4)]}{(1+\lambda\overline{\lambda})^5},\\
					\der{x_4(s(\mu,\lambda))}{\overline{\mu}}=\dfrac{1}{(1+\lambda\overline{\lambda})^4}.
                			\end{array}
              			\right.
			\end{align*}
			\end{center}
			
		Substituting these into the equation for the pullback, we obtain the expressions stated in the  proposition.

	\end{proof}


	We now finish this section with a result concerning the behaviour of the $\alpha$-surfaces with respect to the real structure on $\bb T$. More specifically, for $z\in\bb T$, we want to compare $P_z$ with $P_{\tau(z)}$, where $\tau$ is the real structure in $\bb T$. With this intention we shall state the following results whose proofs follow by straightforward calculations and shall not be done here.
	
	\begin{lem}\label{lem:orientation}
		Let $\lambda\in U\cap U'=\bb CP^1\setminus\{\infty,0\}$ and write $\dfrac{\overline{\lambda}}{\lambda}=x+iy$. The change of basis matrix from the basis $\{v_0^\lambda,v_1^\lambda,v_2^\lambda,v_3^\lambda,v_4^\lambda\}$ to $\{v_0^{(-1/\overline{\lambda})},v_1^{(-1/\overline{\lambda})},v_2^{(-1/\overline{\lambda})},v_3^{(-1/\overline{\lambda})},\\v_4^{(-1/\overline{\lambda})}\}$ is given by
		
		\begin{align}
		\left(
		\begin{array}{ccccc}
		x^2-y^2 & 0 & 0 & 0 & -2xy\\
		0 & x & 0 &-y & 0\\
		0 & 0 & 1 & 0 & 0\\
		0 & -y & 0 & -x & 0\\
		-2xy & 0 & 0 & 0 & -(x^2-y^2)
		\end{array}
		\right).
		\end{align}
	\end{lem}
	
	\begin{cor}\label{cor:orientation}
		Under the notation of the lemma above, the change of basis from $\{v_1^\lambda,v_2^\lambda,v_3^\lambda\}$ to $\{v_1^{(-1/\overline{\lambda})},v_2^{(-1/\overline{\lambda})},v_3^{(-1/\overline{\lambda})}\}$ is given by:
		\begin{align}
		\left(
		\begin{array}{ccc}
		 x & 0 &-y \\
		 0 & 1 & 0 \\
		 -y & 0 & -x \\
		\end{array}
		\right).
		\end{align}
	In particular, the $\alpha$-surfaces corresponding to $z\in\bb T$ and $\tau(z)$ are the same $3$-dimensional affine subspaces of $\bb R^5$ with the reverse orientation.
	\end{cor}
	
	We conclude this section by stating the twistor correspondence between $\bb R^5$ and $\bb T$:
	
	\textit{Every point $(x_0,x_1,x_2,x_3,x_4)\in\mathbb{R}^5$ corresponds to the real section $p(\xi)=(x_0+ix_4)+(x_1+ix_3)\xi+x_2\xi^2-(x_1-ix_3)\xi^3+(x_0-ix_4)\xi^4\in H^0(\mathbb{C}P^1,\mathcal{O}(4))$. Conversely, every point $z\in \bb T$ corresponds to an oriented $3$-dimensional affine subspace of $\bb R^5$ given explicitly in local coordinates by proposition \eqref{prop:twistorplanes} and whose orientation is given by the orientation of the basis $\{v_1^\lambda,v_2^\lambda,v_3^\lambda\}$.} 
	
\section{Bogomolny pairs on GHC-manifolds}\label{section2}

Let $M$ be a regular GHC-manifold whose twistor space is $Z$ and consider the double fibration for the complexified GHC-manifold:
\begin{align*}
	Z\xleftarrow{\eta}Y=\bb CP^1\times M^\bb C\xrightarrow{p}M^\bb C.
\end{align*}
Also, let $\Omega^*_\eta$ be the sheaf on $Y$ of $\eta$-vertical holomorphic forms and define the relative differential operator $d_\eta$ to be the composition map:
\begin{align*}
	\Omega^0(Y)\xrightarrow{d}\Omega^1(Y)\xrightarrow{proj.}\Omega^1_\eta;
\end{align*}
observe that $d_\eta$ annihilates $\eta^*\Omega^0(Z)$.

We shall now state and prove the following lemma:

\begin{lem}\label{lem:flatrelative}
	Let $F$ be a holomorphic bundle on $Z$. Then $d_\eta$ extends to a \textit{flat relative connection} on $\eta^*F$, this is to say, an operator $$ \nabla_\eta:\eta^*F\to\Omega^1\otimes F,$$
	satisfying the Leibniz rule $$ \nabla_\eta(fs)=f\nabla_\eta (s)+d_\eta f\otimes s.$$
	Conversely, if $\eta$ has simply connected fibres, then the holomorphic bundles on $Y$ arising from pull-back of a bundle on $Z$ are those which admit a flat relative connection.
\end{lem}
\begin{proof}
	Suppose $F$ has rank $k$, let $U$ and $U'$ be open sets on $Y$ and $\{e_0,\cdots,e_k\}$ and $\{e'_0,\cdots,e'_k\}$ be local frames for $\eta^*F$ on $U$ and $U'$ respectively. Let $g_{ij}$ be the transition function of $\eta^*F$ from $U'$ to $U$, this is to say, $e_i=\sum^k_{j=0}g_{ij}e'_j$, such that $d_\eta(g_{ij})=0$.
	
	Let $s=\sum^k_{i=0}f_i\otimes e_i$ be a local section for $\eta^*F$ on $U$ define $\nabla_\eta$ on this open set by: $$ \nabla_\eta(s)=\sum^k_{j=0}d_\eta(f_i)\otimes e_i.$$
	If we define it similarly for other trivialisations, we shall prove it is well defined.
	
	In fact, on $U\cap U'$ we can write $s=\sum^k_{i=0}f_ig_{ij}\otimes e'_i$, therefore we have:  $$ \nabla_\eta(s)=\sum^k_{j=0}d_\eta(f_i)g_{ij}\otimes e'_i=\sum^k_{j=0}d_\eta(f_i)\otimes e_i,$$
since $d_\eta(g_{ij})=0$. Then, $\nabla_\eta$ is well-defined and clearly satisfies the Leibniz rule.

	Conversely, Let $E$ be a bundle on $Y$ endowed with a flat relative connection $\nabla_\eta$. Since $\eta$ has simply connected fibres, we can trivialise $E$ with relative parallel section, this is to say, we can find local frames $\{e_0,\cdots,e_k\}$ for $E$ such that $\nabla_\eta(e_j)=0$. Now it is easy to see that the transition functions $g$ for this trivialisation must satisfy $d_\eta(g)=0$, this means that $g$ is constant along the fibres of $\eta$. Thus, each transition function factors as $g=\eta\circ h$, where $h$ is the transition function for a holomorphic bundle $F$ on $Z$.
	
\end{proof}

 Suppose now that a holomorphic bundle $F$ on $Z$ is trivial on each section of $Z$, then the pull-back $\eta^*F$ is trivial on each fibre of $p$. Therefore, $ \hat F=p_*\eta^*F$ is a vector bundle on $M^\bb C$ with the same rank as $F$. Moreover, from the lemma above, the relative flat connection $\nabla_\eta$ can be pushed down via $p$ to an operator $$ D:\hat F\to p_*\Omega^1_\eta\otimes\hat F,$$ satisfying the Leibinitz rule $$ D(fs)=fD(s)+p_*d_\eta(f)\otimes s.$$
 
 We now use a fact from the last section that there exists a canonical isomorphism $$ (p_*\Omega^1_\eta)_x\cong H^0(\bb CP^1_x,\mcal K^*),$$ where $\mcal K_q$ is the subspace of $T_xM\times\bb CP^1$ given by the kernel of the highest weight $1$-forms for each $q\in\bb CP^1$. This isomorphism allows us to define a canonical map 
 \begin{equation}\label{eq:mape}
 e_q: p_*\Omega^1_\eta\to\mcal K^*_q,
 \end{equation}
  given by evaluation at $q\in\bb CP^1$.
 
Restrict $\hat F$ and the operator $D$ to the submanifold $p(\eta^{-1}(z))$ of $M^\bb C$, where $z\in Z_q$ is a point in the fibre of $Z$ at $q\in CP^1$. Since $\nabla_\eta$ is relatively flat, $D$ is a flat connection restricted to this submanifold. Conversely, notice that if we have a bundle $\hat F$ with an operator $D$ that is flat on $p(\eta^{-1}(z))$ for all $z\in Z$, then we obtain a bundle $p^*(\hat F)$ endowed with a relative flat connection $\nabla_\eta=p^*(D)$.

Now consider the splitting $$ p_*\Omega^1_\eta=\Omega^1(M^\bb C)\oplus (E^*\otimes H').$$ Moreover, it is proved in \cite{SU2} that we have $p_*d_\eta=d\oplus0$ under the splitting above. This means that $D$ can be written as $D=\nabla\oplus\Phi$, where $\nabla$ is an actual connection and $\Phi$ is a section of $End(\hat F)\otimes(E^*\otimes H')$ and is called the \textit{Higgs field}. Moreover, on each $\alpha$-surface $\Pi_z=p(\eta^{-1}(z))$ by the composition:
\begin{align*}
	\hat F\xrightarrow{\nabla\oplus\Phi}(\hat F\otimes E^*\otimes H^*)\oplus(\hat F\otimes E^*\otimes H')=\hat F\otimes E^*\otimes\hat H\xrightarrow{e_q}\hat F\otimes\Omega^1(\Pi_z).
\end{align*}

This motivates the following definition:

\begin{defi}
	Let $M$ be a regular  GHC-manifold and $\hat F$ a vector bundle on $M^\bb C$, a \textit{Bogomolny pair} on $\hat F$ is a pair $(\nabla,\Phi)$, where $\nabla$ is connection on $\hat F$ and $\Phi$ is a section of $End(\hat F)\otimes(E^*\otimes H')$, such that the connection $\nabla\oplus\Phi$, as defined by the composition above, is flat on each $\alpha$-surface. Applying the reality condition gives Bogomolny pairs on $M$.
\end{defi}

We have then the Hitchin-Ward correspondence for GHC-manifolds:

\begin{thm}\label{thm:GHChitchinward}
	\cite{SU2}Let $M$ be a regular GHC manifold. There is a one to one and onto correspondence between Bogomolny pairs  $(\nabla,\Phi)$  for a bundle $\hat F$ on $M^\bb C$ and holomorphic bundles $F$ on $Z$ that are trivial on sections. The correspondence remains true in the presence of a real structure.
\end{thm}

\begin{rmk}
	The theorem above gives a Bogomolny pair for the group $SL(n,\bb C)$, where $n$ is the rank of $F$. By considering bundles $F$ on $Z$ whose structure group reduces we have the above correspondence between those bundles and Bogomolny pairs $(\nabla,\Phi)$ for a gauge group $G\subset SL(n,\bb C)$. The objective of this section is to describe Bogomolny pairs for the group $SU(2)$ when $M=\bb R^5$.
\end{rmk}

\begin{rmk}
	In \cite{MG}, Hitchin proves a correspondence between solutions to the Bogomolny equation in $\bb R^3$ and holomorphic bundles on the total space of the holomorphic tangent bundle $\bb T_2$ to $CP^1$ that are trivial on real sections. Therefore, $(\nabla,\Phi)$ is a Bogomolny pair in $\bb R^3$ if and only if it satisfies the Bogomolny equation $F_\nabla=*D_\nabla\Phi$.
\end{rmk}


\subsection{The map $e_q$ and the Higgs field}

We shall now turn our attentions to the case where $M=\bb R^5$. In the last section we saw that the map $e_q$, given by equation \eqref{eq:mape}, plays a very important role in the description of Bogomolny pairs. In this section we shall describe it in the case $M=\bb R^5$.

We know that $p_*(\Omega^1_\eta)=H^*\oplus H'$, therefore, $e_q$ is an equivariant map
\begin{equation*}
	e_q:(\bb C^5)^*\oplus \bb C^3\to K^*_q.
\end{equation*}
In this section we shall describe the real version of this map:
\begin{equation}
	e_q:(\bb R^5)^*\oplus H'_\bb R\to (K^*_\bb R)_q.
\end{equation}

According to \cite{SU2}, under the splitting above, the map $e_q$ acts on $\bb R^5$ as a projection and on $H'_\bb R$ it is described in the sequence \eqref{eq:sequence2}. Then, we shall move towards the description of $$e_q: H'_\bb R\to K^*_q$$ and its real version.

First we shall decompose $\bb C^3$ in weights with respect to $\lambda\in\bb CP^1$. Similarly to the the $\bb C^5$ case, defining $\bb C^3$ as polynomials of degree $2$ in the variable $\xi$ allows us to write the action of $SL(2,\bb C)$ in $\bb C^3$ by:
\begin{align*}
	g\cdot p(\xi)=(c\xi+d)^2 \cdot p\left( \frac{a\xi+b}{c\xi+d}\right),
\end{align*}
where $g=\begin{pmatrix} a &  b\\  c & d \end{pmatrix}\in SL(2,\bb C)$ and $p(\xi)\in\bb C^3$. Moreover, for a point $\lambda\in U\subset\mathbb{C}P^1$ define $g_\lambda\in SU(2)$ by $g_\lambda=\frac{1}{\sqrt{1+\overline{\lambda}\lambda}}\begin{pmatrix} 1&\lambda \\  -\overline{\lambda} & 1  \end{pmatrix}$. Thus, the weight decomposition of $\bb C^3$ with respect to $\lambda$ is:

\begin{align}\label{eq:decompositionc3}
\left\{
\begin{array}{ll}
	\alpha_1^\lambda=g_\lambda^{-1}\cdot 1&=\dfrac{1}{(1+\lambda\overline{\lambda})}(\overline{\lambda}\xi+1)^2\\&=\dfrac{1}{(1+\overline{\lambda}\lambda)}(1+2\overline{\lambda}\xi+\overline{\lambda}^2\xi^2),\\
	\alpha_2^\lambda=g_\lambda^{-1}\cdot\xi&=\dfrac{1}{(1+\lambda\overline{\lambda})}(\overline{\lambda}\xi+1)(\xi-\lambda)\\&=\dfrac{1}{(1+\overline{\lambda}\lambda)}(-\lambda+(1-\lambda\overline{\lambda})\xi+\overline{\lambda}\xi^2),\\
	\alpha_3^\lambda=g_\lambda^{-1}\cdot\xi^2&=\dfrac{1}{(1+\lambda\overline{\lambda})}(\xi-\lambda)^2\\&=\dfrac{1}{(1+\overline{\lambda}\lambda)}(\lambda^2-2\lambda\xi+\xi^2).
\end{array}
\right.
\end{align}

Write $(H')^*=\bb C^3$, $G=SL(2,\bb C)$ and $B$ the Borel subgroup of upper diagonal matrices. Then, the $\alpha_j^\lambda$ trivialise the homogeneous bundle $(G\times_B H')^*$ over $\bb CP^1$, where $\lambda$ is a local holomorphic coordinate for $q\in\bb CP^1$. Then, from proposition \eqref{prop:emap} we know that there is an isomorphism $(K/S)_\lambda^*\to (H')^*$. This isomorphism allows us to describe a global frame for $(K/S)_\lambda^*$:
\begin{align}\label{eq:decompositionks}
\left\{
\begin{array}{ll}
	F_1^\lambda=\dfrac{1}{(1+\overline{\lambda}\lambda)}(W_1^\lambda+2\overline{\lambda}W^\lambda_2+\overline{\lambda}^2W^\lambda_3),\\
	F_2^\lambda=\dfrac{1}{(1+\overline{\lambda}\lambda)}(-\lambda W_1^\lambda+(1-\lambda\overline{\lambda})W^\lambda_2+\overline{\lambda}W^\lambda_3),\\
	F_3^\lambda=\dfrac{1}{(1+\overline{\lambda}\lambda)}(\lambda^2W_1^\lambda-2\lambda W^\lambda_2+W^\lambda_3),
\end{array}
\right.
\end{align}
where $W_1^\lambda=\omega^\lambda_1+i\omega^\lambda_3$, $W^\lambda_2=\omega^\lambda_2$ and $W^\lambda_3=\omega^\lambda_1-i\omega^\lambda_3$, where the $\omega^\lambda_j$ were defined in the last section.
	
	\begin{rmk}
		Before proceeding, it is important to notice that $e_q^*:End(E)\otimes (H')^*\to (K/S)^*$ is given by $$e_q(\phi_1,\phi_2,\phi_3)=\sum_{j=1}^3\phi_jF^q_j.$$
	\end{rmk}


We can now apply the reality condition on $(\bb C^5)^*$ to explicitly describe a global frame for $(K/S)^*_{\bb R}$:
\begin{align}\label{eq:decompositionr3}
\left\{
\begin{array}{ll}
	h_1^\lambda&=F_1^\lambda-F_3^\lambda=\dfrac{1}{(1+\lambda\overline{\lambda})}[(1-\lambda^2)W_1^\lambda+2(\lambda+\overline{\lambda})W_2^\lambda-(1-\overline{\lambda}^2)W_3^\lambda],\\
	h_2^\lambda&=F_2^\lambda=\dfrac{1}{(1+\lambda\overline{\lambda})}[-\lambda W_1^\lambda+(1+\lambda\overline{\lambda})W_2^\lambda+\overline{\lambda}W_3^\lambda],\\
	h_3^\lambda&=i(F_1^\lambda+F_3^\lambda)=\dfrac{i}{(1+\lambda\overline{\lambda})}[(1+\lambda^2)W_1^\lambda+2(\lambda-\overline{\lambda})W_2^\lambda+(1+\overline{\lambda}^2)W_3^\lambda].
\end{array}
\right.
\end{align}

The proposition below describes the map $e_q$ and follows from the discussion above and proposition \eqref{prop:emap}:

\begin{prop}\label{prop:emapex}
	Let $E$ be a vector bundle over $\bb R^5$, $\nabla$ a connection on $E$ and $\Phi=(\phi_1,\phi_2,\phi_3)$ a section of $End(E)\otimes \bb C^3$. On the $\alpha$-surface $P_{(\lambda,\mu)}$ we have: $$ e_q(\nabla\oplus\Phi)|_{P_{(\lambda,\mu)}}=\nabla|_{P_{(\lambda,\mu)}}+\sum_{j=1}^3\phi_jh^\lambda_j. $$
\end{prop}

We conclude this section by mentioning how the results of this section give a natural orientation for the $\alpha$-surfaces. A straightforward calculation proves the following lemma:

\begin{lem}\label{lemma:naturalorientation}
	Let $\lambda\in\bb CP^1$ and $-1/\overline{\lambda}$ be its antipodal, then $h_j^{-1/\overline{\lambda}}=-h_j^\lambda$, for $j=1,2,3$.
\end{lem}

The following corollary says that a choice of frame for the homogeneous bundle $(K/S)^*$ naturally defines an orientation on the $\alpha$-surfaces:

\begin{cor}\label{cor:naturalorientation}
	Let $P_{(\lambda,\mu)}$ be an $\alpha$-surface. Define its orientation by the $3$-form $$\Xi_{(\lambda,\mu)}=h_1^\lambda\wedge h_2^\lambda\wedge h_3^\lambda.$$ Then, $P_{\tau(\lambda,\mu)}$ and $P_{(\lambda,\mu)}$ are the same submanifold of $\bb R^5$ with reverse orientation. 
\end{cor}

\begin{rmk}

	Notice that, from \ref{eq:decompositionr3}, this orientation coincides with the one given by the order of the triple $v_1^\lambda, v_2^\lambda, v_3^\lambda$.
\end{rmk}


\subsection{SU(2)-Bogomolny pairs on $\bb R^5$}


We begin this section by defining  the \textit{fundamental forms}:

\begin{defi}\label{fundamentalforms}
	Consider $h^\lambda_j$, for $j=1,2,3$, as a $1$-form on $\bb CP^1\times\bb R^5$. Let $s$ be the section of $\eta:\bb CP^1\times\bb R^5\to\bb T$ defined in \eqref{naturalsection} . Define the \textit{fundamental forms} on $\bb T$ by $\Psi _j=s^*h^\lambda_j$.
\end{defi}

In our local coordinates we have the following lemma:

\begin{lem}\label{fundamentalforms}
	In local coordinates for the open set $U_0\subset\bb T$, the fundamental forms are given by:
	\begin{align*}
	\left\{
	\begin{array}{ll}
		\Psi_1=&6\mu\dfrac{(1-\overline{\lambda}^2)}{(1-\overline{\lambda}\lambda)^4}d\overline{\lambda},\\
		\Psi_2=&-6\mu\dfrac{\overline{\lambda}}{(1-\overline{\lambda}\lambda)^4}d\overline{\lambda},\\
		\Psi_3=&-6i\mu\dfrac{(1+\overline{\lambda}^2)}{(1-\overline{\lambda}\lambda)^4}d\overline{\lambda}.
	\end{array}
	\right.
	\end{align*}
\end{lem} 
\begin{proof}
	The result follows from proposition \eqref{naturalforms} and from substituting $h_j^\lambda$ in proposition \eqref{eq:decompositionr3}. Moreover, notice that $W_1^\lambda=\omega^\lambda_1+i\omega^\lambda_3$ and $W_3^\lambda=-\omega^\lambda_1+i\omega^\lambda_3$, therefore $s^*W_1^\lambda=0$ and $s^*W_3^\lambda=-6\dfrac{\mu}{(1-\overline{\lambda}\lambda)^4}d\overline{\lambda}$.
\end{proof}

\begin{rmk}
\begin{enumerate}
	\item The fundamental forms will play an important role in the explicit description of the holomorphic structure of the bundle corresponding to a Bogomolny pair on $\bb R^5$.
	\item It is important to notice that each $\Psi_j$ defines a cohomology class in $H^1(\bb T,\mcal O)$ and hence, by exponentiation, an element of the Picard group $Pic_0(\bb T)$. The line bundles corresponding to this class shall be explicitly described in the next section.
\end{enumerate}
\end{rmk}


\begin{defi}\label{defi:su2pair}
Let $E$ be an $SU(2)$ vector bundle on $\bb R^5$, this is to say, $E$ has complex rank $2$ and is equipped with a symplectic form and a quaternionic structure. We say that  the pair $(\nabla,\Phi)$ on $E$ is a $SU(2)$ Bogomolny pair if:
\begin{enumerate}
 	\item $\nabla$ and $\Phi=(\phi_1,\phi_2,\phi_3)$ preserve the symplectic form;
	\item For every $\alpha$-surface $P_z$, the connection $\nabla\oplus\Phi$, given in \eqref{prop:emapex}, is flat.    
\end{enumerate}
\end{defi}

We know that  $P_z$ is a leaf of the integral distribution $\{v^q_1,v^q_2,v^q_3\}$. From the previous section we can choose coordinates $\{\chi^z_1,\chi^z_2,\chi^z_3\}$ such that  $d\chi^z_k=h^q_k$. If $A$ is the connection $1$-form for $\nabla$ on $P_z$, then we can write: $$ e^q(\nabla\oplus-i\Phi)|_{P_z} = \sum^3_{k=1}(A_k-i\phi_k)d\chi^z_k.\footnote{The $-i$ here will become clear in the proof of theorem \eqref{thm:hitchinward}. }$$ The zero curvature condition for this connection gives:
\begin{align}\label{zerocurvature}
	F_{kj}+i\nabla_k\phi_j-i\nabla_j\phi_k-[\phi_j,\phi_k]=0,
\end{align} 
where $F$ is the curvature $2$-form for $\nabla$.


Before proceeding to the main result of this section we shall state the following lemma which compares the connections $e^q(\nabla\oplus\Phi)|_{P_z}$ and $e_{\tau(q)}(\nabla\oplus\Phi)|_{P_{\tau(z)}}$:

\begin{lem} \label{lem:antipodal}
	$e_{\tau(q)}(\nabla\oplus\Phi)|_{P_{\tau(z)}}=\nabla-\phi_1h^q_1-\phi_2h^q_2-\phi_3h^q_3$.
\end{lem}
\begin{proof}
	The proof is a straightforward calculation using lemma \eqref{lemma:naturalorientation}.
\end{proof}


\begin{thm}\label{thm:hitchinward}
	Let $E$ be a $SU(2)$ bundle on $\bb R^5$. There is a 1-1 onto correspondence between $SU(2)$ Bogomolny pairs $(\nabla,\Phi)$ and holomorphic bundles $\tilde E$ on $\bb T$ satisfying:
	\begin{enumerate}[(i)]
		\item $\tilde E$ is trivial on real sections,
		\item $\tilde E$ has a symplectic structure,
		\item $\tilde E$ is equipped with a quaternionic structure $\sigma$ covering $\tau$, this is to say, $\sigma$ is an anti-holomorphic linear map $$ \sigma:\tilde E_z\to\tilde E_{\tau(z)},$$ such that $\sigma^2=-id_{\tilde E_z}$.
	\end{enumerate}
\end{thm}
\begin{proof}

We shall prove the conditions to reduce the gauge group to $SU(2)$ and describe the holomorphic structure for the bundle $\tilde E$ explicitly.

	

	Let $(\nabla,\Phi)$ be a $SU(2)$ Bogomolny pair on $E$ consider the double fibration:
\begin{align*}
	\bb T\xleftarrow{\eta}Y=\bb CP^1\times \bb R^5\xrightarrow{p}\bb R^5.
\end{align*}
Let $s$ be the section of $\eta$ as defined in \eqref{naturalsection}. We already know from theorem \eqref{thm:GHChitchinward} that $\tilde E=s^*(p^*E)$ is holomorphic and trivial on real sections of $\bb T$, however we shall describe this holomorphic structure explicitly:

	Define the operator  $\hol:\Omega^0(\bb T,\tilde E)\to\Omega^{0,1}(\bb T,\tilde E)$ by: $$ \hol t=\left((s^*\nabla)t-i\left[\sum^3_{k=1}(s^*\phi_k)t\otimes\Psi_k\right]\right)^{0,1},  $$where $t$ is a section of $\tilde E$. We claim that $\hol$ is a holomorphic structure on $\tilde E$.

	We have to prove that $\hol^2=0$. To simplify our notation, write $$ \hat \nabla=s^*\nabla-i\Omega,$$ where $\Omega=\sum^3_{k=1}s^*\phi_k\otimes\Psi_k$. Observe that $\Omega$ is a section of $\Omega^1\otimes End(\tilde E)$ and this makes $\hat \nabla$ a connection on $\tilde E$. Then, $\hol^2=F_{\hat \nabla}^{0,2}$, where $F_{\hat \nabla}$ is the curvature of $\hat \nabla$. We have:
	\begin{align*}
		F_{\hat \nabla}&=s^*F_\nabla-i(s^*\nabla(\Omega))-\Omega\wedge\Omega\\
				        &=s^*F_\nabla+i\left[\sum^3_{k=1}\Psi_k\wedge s^*(\nabla\phi_k)\right]-i\left[\sum^3_{k=1}(s^*\phi_k)\otimes d\Psi_k\right]-\sum_{j< k}[s^*\phi_j,s^*\phi_k]\Psi_j\wedge\Psi_k.
	\end{align*}
	Now, $F_{\hat \nabla}^{0,2}$ vanishes from the zero curvature condition \eqref{zerocurvature} on every $\alpha$-surface and $d\Psi_k^{0,2}=\hol\Psi_k=0$. This proves that $\hol$ is a holomorphic structure on $\tilde E$.
	

Let $\omega$ be a symplectic structure on $E$. Since $\nabla$ and $-i\Phi$ preserve $\omega$, from the definition of $\hol$ we must have that $ss^*\omega$ is also preserved by $\hol$. Therefore, $\tilde E$ is endowed with a symplectic structure compatible with $\hol$. 


To describe the quaternionic structure, we shall use an alternative description for the fibres of $\tilde E$. Let $z\in\bb T$ and define: $$\tilde E_z;\{t\in\Gamma(P_z,E)|\,\,e_q(\nabla\oplus\Phi)t=0\}.$$

Now $E$ has a quaternionic structure $\sigma$ and let $t\in\tilde E_z$, then $t$ satisfies
\begin{align*}
	\left(\nabla -i\sum^3_{k=1}\phi_kh^q_k\right)t=0
\end{align*}
Applying $\sigma$:
\begin{align*}
	\left(\nabla +i\sum^3_{k=1}\phi_kh^q_k\right)\sigma(t)=0.
\end{align*}
Using lemma \eqref{lem:antipodal}:
\begin{align*}
	\left(\nabla -i\sum^3_{k=1}\phi_kh^{\tau(q)}_k\right)\sigma(t)=0,
\end{align*}
Thus, $t\in E_z$ implies $\sigma(t)\in E_{\tau(z)}$. Therefore, $\sigma: E_z\to E_{\tau(z)}$ is anti-holomorphic and satisfies $\sigma^2=-id_{E_z}$.

	
	For the converse, we just need to observe that both the symplectic structure $\eta^*\omega$ and the quaternionic structure $\eta^*\sigma$ on the bundle $\eta^*(E)$ are compatible with the flat relative connection $\nabla_\eta$ on $\eta^*(\tilde E)$. Furthermore, both structures remain compatible with $D$ on $E=p_*(\eta^*\tilde E)$ when they are pushed down to $\bb R^5$ via $p$ and, therefore $\nabla$ and $\Phi$ are both compatible with the quaternionic and symplectic structures on $E$.
	
\end{proof}


The theorem above is phrased for the group $SU(2)$, however minor modifications in the real structure leads to Bogomolny pairs for other groups.


\subsection{The bundles $L_{(a,b,c)}$}

To illustrate the construction above we shall find the explicit transition functions for the bundles on $\bb T$ that correspond to a trivial $U(1)$ Bogomolny pair corresponding to the following data: $E=\bb R^5\times \bb C$, $\nabla=d$ and $\Phi=(-ia,-ib,-ic)$, where $a,b,c$ are real numbers, not all vanishing. 

Let $\tilde L$ be the trivial complex line bundle on $\bb T$. From theorem \eqref{thm:hitchinward} we can endow $\tilde L$ with a holomorphic structure $\hol$ given by: $$\hol(s)=\der{s}{\overline{\lambda}}+\Omega(s), $$ where $\Omega=\sum_{j=1}^3-i\phi_j\Psi_j$.

Let $l$ be a smooth trivialisation for $\tilde L$, i.e. $l$ is a non-vanishing complex function on $\bb T$, a local section $s=fl$ is holomorphic if and only if $\hol(fl)=0$. But this means that: $$\der{f}{\overline{\mu}}=0 $$ and $$ \der{f}{\overline{\lambda}}=f\beta,$$ where $\Omega=\beta d\overline{\lambda}$.

Suppose that $f=g\cdot exp(u)$, with $g$ holomorphic, then $$\der{f}{\overline{\lambda}}=\der{u}{\overline{\lambda}}g.$$
Thus, if we want to trivialise $\tilde L$ in a given open set, we have to find a function $u$, regular on this open set, such that $\der{u}{\overline{\lambda}}=\beta$. In this case, $f=g\cdot exp(u)$ will be the given trivialisation.

We shall investigate three separate cases:

\begin{itemize}
\item $\phi_1=\dfrac{i}{2}$, $ \phi_2=0$ and $\phi_3=0$.

	The bundle corresponding to this data will be denoted by $L_{(\frac{1}{2},0,0)}$ In this case, we must have $\Omega=\Psi_1=3\mu\dfrac{(1-\overline{\lambda}^2)}{(1-\overline{\lambda}\lambda)^4}d\overline{\lambda}$. Then 
	$$\beta_1=3\mu\dfrac{(1-\overline{\lambda}^2)}{(1-\overline{\lambda}\lambda)^4}.$$
	Define $$\tilde u_1=-\dfrac{\mu}{(1-\overline{\lambda}\lambda)^3}\left(\dfrac{1}{\lambda}+\overline{\lambda}^3\right)$$
	and observe that $\tilde u_1$ is singular at $\infty$ and at $0$.
	Now, define $\tilde g_1=\dfrac{\mu}{\lambda}$ and 
	$$u_1=\tilde u_1 +\tilde g_1=\dfrac{\mu}{(1-\overline{\lambda}\lambda)^3}\left(3\overline{\lambda}+\lambda\overline{\lambda}^2+\lambda^2\overline{\lambda}^3\right).$$
	Then, since $u_1$ is regular at $0$ and singular at $\infty$, $f_0=exp(u_1)$ defines a trivialisation of $L_{(\frac{1}{2},0,0)}$ in the open set $U_0$.
	
	Now define $\tilde{\tilde g}_1=\dfrac{\mu}{\lambda^3}$. Write $g_1=-\tilde g_1+\tilde{\tilde g}_1$ We have:
	$$u_1+g=\dfrac{\mu}{(1-\overline{\lambda}\lambda)^3} \left(\frac{1}{\lambda^3}+\dfrac{\overline{\lambda}}{\lambda^2}+\dfrac{\overline{\lambda}^2}{\lambda}\right).$$
	Since $u_1+g_1$ is regular at $\infty$ and singular at $0$, $f_1=exp(u_1+g_1)$ is a trivialisation of $L_{(\frac{1}{2},0,0)}$ over $U_1$. 
	On the intersection $U_0\cap U_1$ we have $f_1e^{g_1}=f_0$. Then the transition function for $L_{(\frac{1}{2},0,0)}$ is given by 
	\begin{align}\label{eq:transition1}
	g^1_{01}=exp\left(-\mu\left(\dfrac{1}{\lambda}-\dfrac{1}{\lambda^3}\right)\right).
	\end{align}
	
\item$\phi_1=0$, $ \phi_2=i$ and $\phi_3=0$.

	We shall denote the bundle corresponding to this data by $L_{(0,1,0)}$ and in this case we have $$ \Omega=\Psi_2=-6\mu\dfrac{\overline{\lambda}}{(1-\overline{\lambda}\lambda)^4}d\overline{\lambda}.$$
	Define $$u_2= \dfrac{\mu}{(1-\overline{\lambda}\lambda)^3}\left(\dfrac{3\overline{\lambda}}{\lambda}+\dfrac{1}{\lambda^2}\right).$$ We have that $u_2$ is singular at $0$ but regular at $\infty$, therefore $f_1=u_2$ trivialises $L_{(0,1,0)}$ on $U_1$. Now, for $g_2=-\dfrac{\mu}{\lambda^2}$ we have: $$u_2+g_2=-\dfrac{\mu}{(1-\overline{\lambda}\lambda)^3}\left(3\overline{\lambda}+\lambda\overline{\lambda}\right) ,$$
	which is regular at $0$, but singular at $\infty$. Thus, $f_0=u_2+g_2$ trivialises $L_{(0,1,0)}$ on $U_0$. On the intersection we then have $f_0=e^{\mu/\lambda^2}f_1$. Therefore, the transition function of $L_{(0,1,0)}$ is:
\begin{align}\label{eq:transition2}
	g^2_{01}=exp\left(\dfrac{\mu}{\lambda^2}\right).
\end{align}

\item$\phi_1=0$, $ \phi_2=0$ and $\phi_3=\dfrac{i}{2}$

This case is similar to the first one and we shall write the transition function for this bundle without proof:
\begin{align}
	g^3_{01}=exp\left(-i\mu\left(\dfrac{1}{\lambda}+\dfrac{1}{\lambda^3}\right)\right).
\end{align}

\end{itemize}

Now we state:

\begin{prop}\label{prop:transition}
The bundle $L_{(\frac{a}{2},b,\frac{c}{2})}$ has transition function
\begin{align}
	g^{(a,b,c)}_{01}=exp\left(-a\mu\left(\dfrac{1}{\lambda}-\dfrac{1}{\lambda^3}\right)+b\dfrac{\mu}{\lambda^2}-ic\mu\left(\dfrac{1}{\lambda}+\dfrac{1}{\lambda^3}\right)\right).
\end{align}
\end{prop}

Since the real structure in our local coordinates is given by $$\tau(\lambda,\mu)=(-\overline{\mu}/\overline{\lambda}^4,-1/\overline{\lambda}), $$
noting that $\tau$ interchanges $U_0$ and $U_1$ gives us $$\tau(g^{(a,b,c)}_{01})=\left(\overline{g^{(a,b,c)}_{01}}\right)^{-1}.$$

Therefore we have an anti-holomorphic isomorphism $$ \sigma:L_{(\frac{a}{2},b,\frac{c}{2})}\cong\left(L_{(\frac{a}{2},b,\frac{c}{2})}\right)^*.$$

\subsection{Relations with self-duality on $\bb R^8$}

We bring this section to an end by relating the concepts of Bogomolny pairs and Self-duality. We start with a $1$-hypercomplex manifold $M$ and a complex vector bundle $E$ on $M$. Since, there is no Higgs field for a Bogomolny pair, we can say that a connection $\nabla$ is \emph{self-dual}, or \emph{hyperholomorphic} \cite{MV}, if $\nabla$ restricted to the $\alpha$-surfaces is flat.

Remember that we have $TM^\bb C=E_M\otimes H$, which gives a decomposition $$\Lambda^2T^*M^\bb C=(S^2E^*_M\otimes\Lambda^2 H)\oplus(\Lambda^2E^*_M\otimes S^2H).$$ We now state some results from \cite{SU2} and \cite{MV}:

\begin{prop}
	The following are equivalent:
	\begin{enumerate}[(i)]
	\item $\nabla$ is self-dual,
	\item $F_\nabla$ lies in the component $(S^2E^*_M\otimes\Lambda^2 H)$ in the decomposition above,
	\item $F_\nabla$ is $SU(2)$ invariant.
	\end{enumerate}
\end{prop}

	A connection $\nabla$ is called \emph{Yang-Mills} if $\nabla$ is a minimal of the $\emph{Yang-Mills}$ functional:
	\begin{align}\label{YM}
		\mcal Y(\nabla)=\int_M|F_\nabla|^2\text{vol}_g,
	\end{align}
	where $|F_\nabla|^2=\text{Tr}(F_\nabla\wedge *F_\nabla)$ and $\text{vol}_g$ is the volume form on $M$ with respect to $g$.
	
\begin{rmk}
	It is proved in \cite{MV} that if a connection $\nabla$	 satisfies conditions (i), (ii) or (iii) of the proposition above, then it is Yang-Mills.
\end{rmk}
	
	In order for us to explain the relations between self-dual connections and Bogomolny pairs, we shall first recover the following result from \cite{SU2}:
	
	\begin{thm}
	If $M$ be a $k$-hypercomplex manifold, then there exists a hypercomplex manifold $\tilde M$ with a projection $p:\tilde M\to M$ such that the pair $(\nabla,\Phi)$ on a bundle $F$ on $M$ is a monopole if and only if $p^*(\nabla\oplus\Phi)$ on the bundle $p^*F$ on $\tilde M$ is self-dual.
	\end{thm}

	The results above say that Bogomolny pairs in $\bb R^5$ are obtained from self-duality in $\bb R^8$. However, it is important to remark that a self-dual connection $\nabla$ in $\bb R^8$ does not have finite energy \cite{TAU}, this is to say, $\mcal Y(\nabla)$ is not finite. This makes us to believe that the Yang-Mills-Higgs functional on $\bb R^5$, obtained from \eqref{YM} by dimensional reduction, also does not admit finite energy minima.

\section{Algebraic curves and monopoles on $\bb R^5$}\label{section3}


	In this section we describe the method of constructing Bogomolny pairs from algebraic curves on $\bb T$. 

		Let $S\subset\mathcal{O}(4)$ be a compact algebraic curve in the linear system $|\mathcal{O}(4k)|$, this is to say, on the open set $U$, $S$ is defined by the equation
			\begin{align}\label{eq:spectralcurve}
				P(\xi,\eta)=\eta^k+a_1(\xi)\eta^{k-1}+\cdots+a_{k-1}(\xi)\eta+a_k(\xi)=0,
			\end{align}
		where $a_j(\xi)$ is a polynomial of degree $4j$ in $\xi$.
	
	Next we shall discuss how those curves relate to holomorphic bundles on $\bb T$. In this section, we shall write $L=L_{(a,b,c)}$ for any non-zero $(a,b,c)\in\bb R^3$, where $L_{(a,b,c)}$ is defined in \eqref{prop:transition}. Then, there exists a short exact sequence of sheaves:
	\begin{align}\label{seq:spectralseq}
		0\to\mcal O(L^2(-4k))\to\mcal O(L^2)\to\mcal O_S(L^2)\to0.
	\end{align}
	This gives a long exact sequence on cohomology:
	\begin{align}\label{seq:longspectral}
	0&\to H^0(\bb T,L^2(-4k))\to H^0(\bb T,L^2)\to H^0(S,L^2)\\
	  &\to H^1(\bb T,L^2(-4k))\to H^1(\bb T,L^2)\to H^1(S,L^2)\cdots
	\end{align}
	
	Assume further that $S$ is such that $L^2|_S$ is trivial. This implies that $H^1(S,L^2)=0$ and \eqref{seq:longspectral} becomes:
	\begin{align}\label{seq:longspectral2}
	0\to H^0(S,L^2)\xrightarrow{\delta} H^1(\bb T,L^2(-4k))\xrightarrow{\otimes\psi} H^1(\bb T,L^2)\to 0,
	\end{align} 
	where $\psi\in H^0(\bb T,\mcal O(4k))$ is the section defining $S$.
	
	Choose a trivialisation $s$ of $L^2$ over $S$, this is to say, $s$ is a non-zero element in $H^0(S,L^2)$. Define the bundle $\tilde E$ over $\bb T$ by the cohomology class $\delta (s)$. This means $\tilde E$ is given as an extension:
		\begin{align*}
			0\to L(-2k)\xrightarrow{\alpha}\tilde E\xrightarrow{\beta} L^*(2k)\to0.
		\end{align*}

		We then have the following:
		
		\begin{prop}\label{prop:keepinmind}
			 $\tilde E$ satisfies the following conditions:
			 \begin{enumerate}[(i)]
				\item $\tilde E$ has a symplectic structure,
				\item If $S$ is real and $L(-2k)$ has a quaternionic structure on $S$, then $\tilde E$ is equipped with a quaternionic structure $\sigma$ covering $\tau$, this is to say, $\sigma$ is an anti-holomorphic linear map $$ \sigma:\tilde E_z\to\tilde E_{\tau(z)},$$ such that $\sigma^2=-id_{\tilde E_z}$.
			\end{enumerate}
		\end{prop}
		\begin{proof}
			The properties (i) is straightforward from the definition of $\tilde E$. For (ii), let $\sigma:L\to L^*$ be the anti-holomorphic isomorphism. We can define a bundle $\sigma(\tilde E)$ on $\bb T$ via the extension:
			\begin{align*}
				0\to L^*(-2k)\xrightarrow{\alpha'}\sigma(\tilde E)\xrightarrow{\beta'} L(2k)\to0.
			\end{align*}
		Now, we can extend the antiholomorphic isomorphism $L\cong L^*$ to an antiholomorphic isomorphism $\tilde E\cong\sigma(\tilde E)$.
		\end{proof}

	


	We shall need the following facts about these spectral curves \cite{HAH} 
	
	\begin{prop}
		The cohomology group $H^1(\mathbb{T},\mathcal{O}_{\mathbb{T}})$ is generated by $\eta^i /\xi^j,0<i\leq k-1,0<j<ki$.
	\end{prop}
	
	Noticing that $exp:H^1(\mathbb{T},\mathcal{O}_{\mathbb{T}})\to Pic_0(S)$ is an isomorphism, the bundles with vanishing degree over $S$ are generated by $exp(\eta^i /\xi^j)$.
	
	\begin{prop}
		The natural map $H^1(\mathbb{T},\mathcal{O}_{\mathbb{T}})\to H^1(S,\mathcal{O}_S)$ is surjective, which means that $H^1(S,\mathcal{O}_S)$ is generated by $\eta^i /\xi^j,0<i\leq k-1,0<j<ki$.
	\end{prop}
	
	Notice that if $S$ is smooth the proposition above gives degree zero line bundles on $S$. In this section we shall assume the curves are smooth, the adjustments for the non-smooth case essentially follow from what is done in \cite{HAH}.
	
	A bit of notation: Let $\pi:\mathbb{T}\to \mathbb{C}P^1$, then $\mathcal{O}_{\mathbb{T}}(l)$ denotes the pull-back of $\mathcal{O}(l)$ by $\pi$. Also, if $F$ is a sheaf on $\mathbb{T}$ we denote by $F(l)$ the sheaf $F\otimes\mathcal{O}_\mathbb{T}(l)$
	
		
	
	\begin{defi}
		The theta divisor $\Theta$ in $S$ is the set of line bundles of degree $g-1$ that have a non-zero global section. The affine Jacobian $J^{g-1}$ is the set of line bundles of degree $g-1$ on $S$.
	\end{defi}
	
	\begin{thm}
		[Beauville \cite{Beau}] There is a 1-1 correspondence between $J^{g-1}\setminus\Theta$ and $Gl(k,\mathbb{C})$-conjugacy classes of $gl(k,\mathbb{C})$-valued polynomials $A(\xi)=\sum_{j=0}^{4}A_j\xi^j$ such that $A(\xi)$ is regular for every $\xi$ and the characteristic polynomial of $A(\xi)$ is \eqref{eq:spectralcurve}.
	\end{thm}
	
	We shall now give the idea of this construction with enough details that will be necessary when we see the boundary conditions. In order to do this, we need the following lemma:
	
	\begin{lem}\label{lem:techiso}
		Let $E$ be an invertible sheaf on $\mathbb{T}$ whose degree is $g-1$ and such that $H^0(S,E)=0$, then $H^0(S,E(1))\cong\mathbb{C}^k$.
	\end{lem}
	\begin{proof}
		Let $\xi_0\in\mathbb{C}P^1$ and denote by $D_{\xi_0}$ the divisor corresponding to the meromorphic function $(\xi-\xi_0)$ on $S$, this means that as a set $D_{\xi_0}$ consists of the points of $S$ in the fibre $E_{\xi_0}$, which is a set of $k$ points, counted with multiplicities. 
		
		Now consider the exact sequence of sheaves; \cite{GH} page 139.
			\begin{align*}
				0\to\mathcal{O}_S(E)\to\mathcal{O}_S(E(1))\to\mathcal{O}_{D_{\xi_0}}(E(1))\to0.
			\end{align*}
		
		From Riemann-Roch, the hypothesis $H^0(S,E)=0$ implies that $H^1(S,E)=0$.  Taking the exact sequence on cohomology and noticing that $H^0(D_{\xi_0},E)=\mathbb{C}^k$, since $D_{\xi_0}s$ are a set of $k$ points counted with multiplicity, gives the required isomorphism.
	
  	\end{proof}

	For $\xi\in U$, define a map $Z:H^0(D_{\xi},E(1))\to H^0(D_{\xi},E(1))$ given by multiplication by $\eta$. We define the linear map $$A(\xi):\cohm{0}{S}{E(1)}\to\cohm{0}{S}{E(1)}$$ by the commutative diagram:
	\begin{align*}
	\begin{CD}
		\cohm{0}{S}{E(1)} @>>> H^0(D_{\xi},E(1))\\
		@VVA(\xi)V @VVZV\\
		\cohm{0}{S}{E(1)} @>>> H^0(D_{\xi},E(1))
	\end{CD}
	\end{align*}

where the horizontal maps are the isomorphism given in the lemma \eqref{lem:techiso}.

Conversely, let $A(\xi)=\sum_{j=0}^{4}A_j\xi^j$ be a regular matricial polynomial and define a sheaf $E(1)$ over $\bb T$ via the exact sequence: 
\begin{align}\label{eq:matricestolinebundles}
	0\to\mcal O(-4)_\bb T^{\oplus k}\xrightarrow{\eta-A(\xi)}\mcal O_\bb T^{\oplus k}\to E(1)\to0.
\end{align}

$E(1)$ is supported on $S$ and since $A(\xi)$ has $1$-dimensional kernel, $E(1)$ is a line bundle of degree $g-1$, where $g$ is the genus of $S$. 

\subsection{Monopoles from spectral curves} 

From now on in this paper, we shall consider the bundle $L$ to be the line bundle on $\bb T$ given by the transition function $g=exp(\eta/\xi^2)$. Now we define spectral curves:

	\begin{defi}\label{spectral}
	A spectral curve is a compact algebraic curve $S$ in $\mathbb{T}$ satisfying:
		\begin{enumerate}[(i)]
		\item $S$ is a compact algebraic curve in the linear system $\mathcal{O}(4k)$, therefore it is given by an equation of the type 
		\begin{align}\label{eq:spectralcurve}
				P(\xi,\eta)=\eta^k+a_1(\xi)\eta^{k-1}+\cdots+a_{k-1}(\xi)\eta+a_k(\xi)=0,
		\end{align}
		where $a_j(\xi)$ is a polynomial of degree $4j$ in $\xi$.
		\item $S$ has no multiple components.
		\item The line bundle $L$ has order $2$ on $S$.
		\item $H^0(S,L^z(2k-3))=0$ for $z\in (0,2)$.
		\end{enumerate}	
	\end{defi}
	
	Using the results from the last section, we are able to make the following definition:
	
	\begin{defi}
		A Bogomolny pair $(\nabla,\Phi)$ on $\bb R^5$ is called a \emph{monopole} if the corresponding holomorphic bundle $\tilde E$ on $\bb T$ is defined via a spectral curve $S$ satisfying the conditions of definition \eqref{spectral}. We say the \emph{algebraic charge} of the monopole is $k$ if the curve $S$ corresponding to the monopole has degree $k$.
	\end{defi}
	
	We shall call the algebraic charge shortly by charge, however we must bear in mind that we do not have yet a topological definition of charge for monopoles in $\bb R^5$.

	\begin{rmk}
		Our main motivation to make this definition is that the conditions in \eqref{spectral} are similar to the conditions for the spectral curves for monopoles in $\bb R^3$ \cite{CM}. Also, condition (iv) above allows us to define a flow of endomorphisms for a bundle over the interval $(0,2)$ whose fibre at $z\in(0,2)$ is $\cohm{0}{S}{L^z(2k-2)}$.
	\end{rmk}	
	
	\begin{eg}\label{eg:k=1}
		A spectral curve $S$ for $k=1$ is given by the equation $\eta+a_1(\xi)=0$, where $a_1$ is a polynomial of degree $4$. Imposing the condition that $S$ is real gives $\overline a_1(\xi)=\overline\xi^4a_1(-\frac{1}{\overline\xi})$, but this condition says that $S$ is a real section $P$ of $\bb T$ over $\bb CP^1$. Moreover, conditions $(ii)$ and $(iii)$ are clearly satisfied, remember that $L$ is trivial on real sections of $\bb T$ since it corresponds to a bogomolny pair on $\bb R^5$. For condition $(iv)$, notice that on $P$, we have $L^z(2k-3)=L^z(-1)\cong\mcal O(-1)$. Then, we conclude that the spectral curves for charge $1$ monopoles correspond to real sections of $\bb T$.
	\end{eg}
	

	
	
\subsection{From linear flows to Nahm's equations}
		
	Also, in the last section, we saw that in order for us to construct the monopole data we need a antiholomorphic isomorphism $L=L^*$ on $S$ and this implies that $L^2$ is trivial on $S$. This condition and the condition (iii) of \eqref{spectral} above implies that the element $g\in H^1(S,\mathcal{O})$  is a lattice point in $H^1(S,\mathbb{Z})$. Thus, the straight line between $0$ and $g$ defines a morphism, which we will refer as a flow:
	\begin{align*}
		h:S^1&\to H^1(S,\mathcal{O})/H^1(S,\mathbb{Z})\cong Pic^0(S)\\
		       exp(i\pi z)&\mapsto exp(izg), \,\, z\in[0,2].
	\end{align*}
		Let $Pic^0(S)$ be the group of degree $0$ line bundles on $S$ and $J^{\mathbf{g}-1}(S)$, the Jacobian of line bundles of degree $\mathbf g-1$, where $\mathbf g$ is the genus of $S$. We can identify $Pic^0(S)$ with $J^{\mathbf g-1}(S)$ by doing $F\to F(2k-3)$, since $deg(F(2k-3))=k(2k-3)=2k^2-3k=(k-1)(2k-1)-1$. Now, $h$ can be considered a flow in the Jacobian and the condition (iv) in \eqref{spectral} says that, for $z\in(0,2)$, $h(z)$ is not in the theta divisor, the line bundles in $J^{\mathbf g-1}(S)$ with a non-vanishing holomorphic section.
		
	These properties will allow us to derive the Nahm's equations satisfying the appropriate boundary conditions. However, the boundary conditions will be treated in the next section. 
		
	\begin{lem}
		For $z\in (0,2)$ we have $dimH^0(S,L^z(2k-2))=k$.
	\end{lem}
	\begin{proof}
	This lemma follows from lemma \ref{lem:techiso} by noticing that the degree of $L(2k-3)$ as a line bundle on $S$ is $\mathbf g-1$, where $\mathbf g=(k-1)(2k-1)$ is the genus of $S$.
	\end{proof}
	
	We can now define a bundle $V$ on $\mathbb{C}$ in the following way: Let $W$ be the bundle over $\mathbb{C}\times S$ whose fibre at $(z,p)$ is $L^z(2k-2)_p$ and $P_1:\mathbb{C}\times S\to\mathbb{C}$ be the projection in the first coordinate, define $V=(P_1)_*W$.\footnote{V is a locally free sheaf since the direct image sheaf $(P_1)_*W$ over $\mathbb{C}$ is torsion free.} From the proposition above, we know that $V$ has rank $k$ and, moreover, the fibre at $z\in(0,2)$ is $V_z=H^0(S,L^z(2k-2))$. 
	
Now we shall state the following lemma whose proof is similar to the proof of proposition (4.5) in \cite{CM}.

\begin{lem}\label{lem:tecreal}
	If $l < 4k$, then any section $s\in\cohm{0}{S}{\mcal O(l)}$ can be written uniquely as:
	\begin{align*}
		s=\sum_{j=0}^{[l/4]}\eta^j\pi^*(c_j),
	\end{align*} 
	where $c_j\in\cohm{0}{\bb CP^1}{l-4j}$.
\end{lem}
	
	Observe that this lemma implies that at $z=0$ the bundle $L^z(2k-2)$ has more sections than for $z\in(0,2)$. This means that the fibre $V_0$ is not just $\cohm{0}{S}{L^z(2k-2)}$ and we shall treat this case later. In this section we shall consider the behaviour of the bundle $V$ on the interval $(0,2)$ and the endpoints will be studied in the next section.
	
	From Beauville's theorem we have that, for $z\in(0,2)$, each line bundle $L^z(2k-3)$ corresponds to a conjugacy class of a regular matricial polynomial $A(\xi,z)=\sum_{j=0}^{j=4}A_j(z)\xi^j$. Moreover, $A(\xi,z)$ can be seen, from its construction, as a linear map $A(\xi,z):H^0(S,L^z(2k-2))\to H^0(S,L^z(2k-2))$, this is to say, $A(\xi,z):V_z\to V_z$. However, we want to define actual matrices, and so far we only have an equivalence class of matrices, in other words, we have endomorphisms of $V_z$. The objective now is to use the endomorphisms $A_j(z)$ to define a connection for  $V$ on the interval $(0,2)$. Then we shall trivialise $V$ by parallel constant section with respect to this connection.
	
	Let $s(z)$ be a local holomorphic section of $V$, we can write it as a pair of holomorphic functions $f_0:S\cap U_0\times\bb C^*\to\mathbb{C}^k$ and $f_1:S\cap U_1\times\bb C^*\to\mathbb{C}^k$ satisfying $f_0=exp(z\eta/\xi^2)\xi^{2k-2}f_1$ on $U_0\cap U_1$. We now follow the construction on \cite{CM} page 169:

	Differentiating with respect to $z$:
	\begin{align*}
		\frac{\partial f_0}{\partial z}=\frac{\eta}{\xi^2}e^{z\eta/\xi^2}\xi^{2k-2}f_1+e^{z\eta/\xi^2}\xi^{2k-2}\frac{\partial f_1}{\partial z}.
	\end{align*}
	
	From the definition of $A$, we have:
	\begin{align*}
		(\eta-A_0-A_1\xi-A_2\xi^2-A_3\xi^3-A_4\xi^4)s=0,
	\end{align*}
	or
	\begin{align*}
		\frac{\eta}{\xi^2}s=(A_0\xi^{-2}+A_1\xi^{-1}+\frac{1}{2}A_2)s+(\frac{1}{2}A_2+A_3\xi+A_4\xi^2)s.
	\end{align*}
	This implies that on $U\cap U'$ we have
	\begin{align*}
		&\frac{\partial f_0}{\partial z}-(\frac{1}{2}A_2+A_3\xi+A_4\xi^2)s\\
		=&\frac{\partial f_0}{\partial z}-\frac{\eta}{\xi^2}f_0+(A_0\xi^{-2}+A_1\xi^{-1}+\frac{1}{2}A_2)s\\
		=&e^{z\eta/\xi^2}\xi^{2k-2}\frac{\partial f_1}{\partial z}+e^{z\eta/\xi^2}\xi^{2k-2}(A_0\xi^{-2}+A_1\xi^{-1}+\frac{1}{2}A_2)s\\
		=&e^{z\eta/\xi^2}\xi^{2k-2}\left[\frac{\partial f_1}{\partial z}+(A_0\xi^{-2}+A_1\xi^{-1}+\frac{1}{2}A_2)s \right].
	\end{align*}
	
	The lines above tell us that we can define a connection on $V$, over $(0,2)$, whose covariant derivative on $U$ is given by:
	\begin{align*}
		\nabla_zs=\frac{\partial f_0}{\partial z}-(\frac{1}{2}A_2+A_3\xi+A_4\xi^2)s.
	\end{align*}
	We shall use this to define a frame $(s_1,\cdots,s_k)$ of covariant sections for $V$.
	
	Let $A_+=\frac{1}{2}A_2+A_3\xi+A_4\xi^2$, then we can write
	\begin{align*}
		\frac{\partial s}{\partial z}-A_+s=0.
	\end{align*}
	
	Taking the derivative of
	\begin{align*}
		(\eta-A)s=0,
	\end{align*}
	with respect to $z$, we have
	\begin{align*}
		(\eta-A)\frac{\partial s}{\partial z}-\frac{\partial A}{\partial z}s=0.
	\end{align*}
	Thus,
	\begin{align*}
		-(\eta-A)A_+s-\frac{\partial A}{\partial z}s=-\eta A_+s+AA_+s-\frac{\partial A}{\partial z}s=0,
	\end{align*}
	hence
	\begin{align*}
		\left([A,A_+]-\frac{\partial A}{\partial z} \right)s=0.
	\end{align*}
	Observe that this equation is independent of $\eta$.
	
	Now let $F$ be a fibre of $\mathbb{T}$ such that $F\cap S=\{x_1,\cdots,x_k\}$ with the $x_j$ all distinct. Therefore, we have an exact sequence
	\begin{align*}
	0\to\mathcal{O}_SL^z(2k-3)\to\mathcal{O}_SL^z(2k-2)\to\mathcal{O}_{F\cap S}\to0.
	\end{align*}
	Using the fact that $H^0(S,L^z(2k-2))=H^1(S,L^z(2k-2))=0$, the exact cohomology sequence says that the restriction map $H^0(S,L^z(2k-2))\to H^0(F\cap S,\mathcal{O})$ is an isomorphism. Thus, we can find a frame $s_1,\cdots,s_k$ for $H^0(S,L^z(2k-2))$ such that $s_i(x_j)=\delta_{ij}$.
	
	We then have that $B=[A,A_+]-\frac{\partial A}{\partial z}$ satisfies $\sum_jB_{ij}s_j(x_l)=0$ for all $i,l$. But this says that $B_{ij}=0$. Since the condition on $F$ is generic, we must have $B_{ij}=0$ for all the fibres. Thus, we must have
	\begin{align*}
		\frac{\partial A}{\partial z}=[A,A_+].
	\end{align*}	
	We therefore have the generalized Nahm's equations:
	\begin{align}\label{eq:cNahm's}
	\begin{array}{ll}
		\dot{A}_0&=\dfrac{1}{2}[A_0,A_2]\\
		\dot{A}_1&=[A_0,A_3]+\dfrac{1}{2}[A_1,A_2]\\
		\dot{A}_2&=[A_1,A_3]+[A_0,A_4]\\
		\dot{A}_3&=[A_1,A_4]+\dfrac{1}{2}[A_2,A_3]\\
		\dot{A}_4&=\dfrac{1}{2}[A_2,A_4],
	\end{array}	
	\end{align}
	for $z\in(0,2)$ and $\dot{A}_j=\dfrac{\partial A_j}{\partial z}$.	
	

	\begin{rmk}\label{rmk:momentmap}
		We now make an important remark on the equations \ref{eq:Nahm's}. Let $k\geq 2$, we defined the endomorphisms as linear maps $A(\xi,z):H^0(S,L^z(2k-2))\to H^0(S,L^z(2k-2))$. Now, in order for $\tilde E$ in \ref{prop:keepinmind} to inherit a quaternionic structure we need $L(-2k)$ to be quaternionic on $S$ and therefore $L(2k-2)=L(-2k)\otimes\mcal O(4k-2)$ is quaternionic. This means that there is no reality condition to be imposed on $A(\xi,z)$ that gives us real spectral curves, ie, this method does not construct monopoles for the group $SU(2)$. However, we can still impose a reality. Namely, let $A_0=T_1+iT_2$, $A_1=T_3+iT_4$, $A_2=2iT_5$, $A_3=T_3-iT_4$ and $A_4=-T_1+iT_2$, then we have:
	\begin{align}\label{eq:Nahm's}
	\begin{array}{ll}
		\dot{T}_1&=[T_5,T_2]\\
		\dot{T}_2&=[T_1,T_5]\\
		\dot{T}_3&=[T_1,T_3]+[T_2,T_4]+[T_5,T_4]\\
		\dot{T}_4&=-[T_1,T_4]+[T_2,T_3]-[T_5,T_3]\\
		\dot{T}_5&=[T_1,T_2]+[T_4,T_3].
	\end{array}	
	\end{align}		
		The interesting fact here is that equations \ref{eq:Nahm's} can be interpreted as a $2$-symplectic moment map \cite{SU2} and it would be interesting to study the moduli space of solutions to \ref{eq:Nahm's}.
	\end{rmk}


	Before proceeding to the next section, we give an alternative description of  the endomorphisms $A_j(z)$ that will be useful later. Let $S$ be a spectral curve and consider the map:
	\begin{align}\label{eq:mapm}
		m:\cohm{0}{S}{\mcal O(4)}\otimes\cohm{0}{S}{L^z(2k-2)}\to\cohm{0}{S}{L^z(2k+2)},
	\end{align}
	and denote by $K_z$ its kernel at $z$. We have the following proposition, whose proof is similar to the proof of proposition 4.8 in \cite{CM}:
	
\begin{prop}\label{prop:kernelm}
The map $h:K_z\to V_z$ given by
	\begin{align*}
		h(\eta\otimes t_0+1\otimes s_0+\xi\otimes s_1+\xi^2\otimes s_2+\xi^3\otimes s_3+\xi^4\otimes s_4)\mapsto t_0
	\end{align*}
is an isomorphism for every $z\in(0,2)$.
\end{prop}

An immediate consequence of this proposition is that there exist endomorphisms $A_j(z)\in \text{End}V_z$ such that : $$ (\eta-A_0-A_1\xi-A_2\xi^2-A_3\xi^3-A_4\xi^4)s=0.$$
The uniqueness of Beauville's theorem tells us that these endomorphism are the same ones obtained via Beauville's theorem for the bundle $L^z(2k-3)$.

\section{Boundary conditions for the Nahm's equations}\label{section4}

	In this section we find necessary and sufficient conditions on the matrices $A_j$ in the equations \ref{eq:cNahm's} such that they correspond to spectral curves satisfying conditions i)-iv) in definition \ref{spectral}.
	

	\begin{defi}
	Let $p(\xi,\eta)$ be the polynomial defining the spectral curve $S$, this is to say, $S=\{(\xi,\eta)|\,\,p(\xi,\eta)=0\}$, we shall use the following notation in this section:
	
		\begin{enumerate}[a)]
			\item Define $M=\bb C\times S$ and $P:M\to \bb C$ is the projection in the first coordinate.
			\item $\tilde{M}=\{(z,\xi,\eta)\in\bb C \times\bb T|\,\,\tilde p(z,\xi,\eta)=0\}$, where $\tilde p(z,\xi,\eta)=z^kp \left( \xi,\dfrac{\eta}{z} \right)$ and $\tilde P: \tilde M\to \bb C$ is the projection in the first coordinate.
			\item For fixed $z\in\bb C$ we define the curve $zS$, it is $S$ shrunk by a factor $z$, to be the curve defined by $\tilde p(z,\xi,\eta)$.
			\item $\tilde V=\tilde P_*(X|_{\tilde M})$, where $X$ is the bundle on $\bb T$ whose fibre at $(z,\eta,\xi)$ is $L^z(2k-2)_{(\eta,\xi)}$.
			\item Define $\mcal L$ over $\bb C\times\bb T$ to be the bundle such that $\mcal L_{\{z\}\times \bb T}=L^z$.

			\item Similarly, we have $X=P_*(\mcal L(2k+2))$ and $\tilde X=\tilde P_*(\mcal L(2k+2))$
			\item Bundles on $\bb T$, their lifts to $\bb C\times \bb T$ and their restrictions to $M$ and $\tilde M$ will be denoted by the same letter.

		\end{enumerate}
	\end{defi}
	\begin{rmk}
		\begin{enumerate}[i)]
			\item If we denote the zero section of $\mcal O(4)$ by $F$, we then notice that $\tilde P^{-1} (0)= F^{(k-1)}$, the $(k-1)\ts{th}$ formal neighbourhood of $F$ in the total space of $\mcal O(4)$.
			\item $V=P_*(X|_M)$ is the bundle defined in the previous section.
		\end{enumerate}	
	\end{rmk}
	
\subsection{The fibre of $V$ at $0$}

	\begin{lem}\label{lem:ztozn}
		Define a map $\rho:L(k)|_{\tilde M}\to\mcal L(k)_M$ in the following way: Let $s$ be a section of $L(k)$ on $\tilde M$ such that on the trivialisation $U_i$ it is represented by $\tilde f_i(z,\eta,\xi)$. Define $\rho(s)$ to be the section of $\mcal L(k)$ on $M$ represented by $f_i(z,\eta,\xi)=\tilde f_i(z,z\eta,\xi)$. Then, $\rho$ is a well defined map of bundles and it is an isomorphism for $z\neq 0$.
	\end{lem}
	\begin{proof}
		We just need to verify that $f_0=\exp(z\eta/\xi^2)f_1$, but this is true since $\tilde f_0=\xi^k\exp(\eta/\xi^2)\tilde f_1$. It is immediate that $\rho$ is an isomorphism.
	\end{proof}
	\begin{cor}\label{cor:ztozn}
		Taking the direct images in the lemma above, there is a map of sheaves over $\bb C$
			\[
				\rho:\tilde V\to V
			\]
		which is an isomorphism for $z\neq0$.
	\end{cor}
	
	Consider now the evaluation map:
	$$\tilde {ev}_z:\tilde V_z\to \cohm{0}{\tilde P^{-1}(z)}{L(2k-2)} .$$
	
	It is an isomorphism for $z\neq 0$.

	For the next result, we shall use the following notation: $\Gamma_m\subset\mcal O(2m)$ consist of sections $s$ of the form $s=\sum_{j=0}^ma_j\xi^{2j}$ and denote by $L\otimes\Gamma\subset L(2m)$ the set of sections of the form $\sum_{jk}\alpha_j\otimes s_k$, with $\alpha_j$ a section of $L$ and $s_k\in\Gamma_m$. 

	The first result in this section is:
	
	\begin{prop}\label{prop:fibreat0}
		Let $V_0\subset H^0(S,\mathcal{O}(2k-2))$ be the fibre of $V$ at $z=0$ and $\Gamma_{k-1}\subset H^0(\mathbb{C}P^1,\mathcal{O}(2k-2))$ be the sections of the form $p(\xi)=c_{2k-2}\xi^{2k-2}+c_{2k-4}\xi^{2k-4}+\cdots+c_2\xi^2+c_0$. Then, $V_0\cong\Gamma_{k-1}$.
	\end{prop}
	
	An extension of a section of $\mcal O(2k-2)$ to a section of $X$ to the $m\ts{th}$ formal neighbourhood consists of the following data:
	\begin{align*}
		s&=s_0+zs_1+\cdots+z^ms_m,\,\,\,\,\,s_i\in\cohm{0}{U_0}{\mcal O},\\
		s'&=s'_0+zs'_1+\cdots+z^ms'_m,\,\,\,\,\,s'_i\in\cohm{0}{U_1}{\mcal O},
	\end{align*}
	such that $s=\xi^{2k-2}(e^{\eta/x_i^2})s'\text{mod}z^{m+1}$ on $U_0\cap U_1$. From lemma \eqref{lem:ztozn} we have that we can change $z$ to $z\eta$ near $z=0$. This means the extension above can be written as:
	\begin{align*}
		p&=p_0+zp_1+\cdots+z^mp_m,\,\,\,\,\,p_i\in\cohm{0}{U_0}{\mcal O(2k-2-4i)},\\
		p'&=p'_0+zp'_1+\cdots+z^mp'_m,\,\,\,\,\,p'_i\in\cohm{0}{U_1}{\mcal O(2k-2-4i)},
	\end{align*}
	such that $p=(e^{\eta/x_i^2})p'\text{mod}\eta^{m+1}$. We can now state and prove the following:

	\begin{lem}\label{lem:extension}
		Every section in $L\otimes\Gamma_m$  on $Z\subset\mathbb{T}$ can be extended uniquely to the $m\ts{th}$ formal neighbourhood, but no section can be extended to the $(m+1)\ts{th}$ neighbourhood.
	\end{lem}
	\begin{proof}
		A section of $L(2m)$ on the m\ts{th} neighbourhood consists of local section $p_i\in H^0(U_0,\mathcal{O}(2m-4i))$ and $p'_i\in H^0(U_1,\mathcal{O}(2m-4i))$, such that 
			\begin{align*}
				p_0+\eta p_1+\cdots+\eta^mp_m=e^{\eta/\xi^2}(p'_0+\eta p'_1+\cdots+\eta^mp'_m)\text{mod}\eta^{m+1}.
			\end{align*}
		We are therefore looking for functions $p_i$ on $U_0$ and $p'_i$ on $U_1$ such that on the intersection $U_0\cap U_1$ we have:
		
			\begin{align}\label{eq:matrix1}
				 	\left(\begin{array}{ccccc}
						\xi^{2m} & 0 & 0 & \ldots & 0 \\
						\xi^{2m-2} & \xi^{2m-4} & 0 & \ldots & 0 \\
						\frac{1}{2}\xi^{2m-4} & \xi^{2m-6} & \xi^{2m-8} & \ldots & 0 \\
						\vdots & \vdots & \vdots & \vdots & \vdots \\
						\frac{1}{m!} & \frac{1}{(m-1)!}\xi^{-2}  & \cdots & \cdots & \xi^{-2m} 
					\end{array} \right)
					\left(\begin{array}{c}
						p'_0 \\
						p'_1 \\
						\vdots \\
						\vdots  \\
						p'_m
					\end{array} \right)=
					\left(\begin{array}{c}
						p_0 \\
						p_1 \\
						\vdots \\
						\vdots  \\
						p_m
					\end{array} \right).			
			\end{align}
			
		Now for $l$ even and such that $0\leq l\leq 2m$,
		
			\begin{align}\label{eq:solution2}
					\left(\begin{array}{c}
						p'_0 \\
						p'_1 \\
						\vdots \\
						\\
						\vdots  \\
						p'_m
					\end{array} \right)=
					\left(\begin{array}{c}
						c_0 \xi^{-2m+l}\\
						c_1 \xi^{-2m+l+2}\\
						\vdots \\
						c_{(\frac{2m-l}{2})}  \\
						\vdots\\
						0
					\end{array} \right)					
			\end{align}
		solves \eqref{eq:matrix1} if
			\begin{align}\label{eq:solution1}
		\sum_{i=0}^{(\frac{2m-l}{2})}\frac{c_i}{(n-i)!}=0,
			\end{align}
			
			where $\left(\dfrac{l}{2}+1\right)\leq n\leq m$.
			
		From \cite{CM} page 173, there exists a unique solution of \eqref{eq:solution1}, and for this solution we have $c_0$ and $c_{(\frac{2m-l}{2})}$ are both non-vanishing. This implies that \eqref{eq:solution2} trivialises a rank-$m+1$ bundle $E_m\to\bb C P^1$ whose transition function is given by the matrix in \eqref{eq:matrix1}.
		
		From the exact sequence
		\begin{align}
			0\to E_{m-1}(-2)\to E_m\xrightarrow{p_0} \mcal O(2m)\to0,
		\end{align}
		we have the following long exact sequence in cohomology:
		\begin{align}
			0&\to\cohm{0}{\bb CP^1}{E_{m-1}(-2)}\to\cohm{0}{\bb CP^1}{E_m}\xrightarrow{p_0}\cohm{0}{\bb CP^1}{\mcal O(2m)}\to\\
			&\to\cohm{1}{\bb CP^1}{E_{m-1}(-2)}\to\dots
		\end{align}
		We can deduce from general sheaf cohomology theory that $\cohm{1}{\bb CP^1}{E_{m-1}(-2)}=0$. Therefore $\cohm{0}{\bb CP^1}{E_m}\xrightarrow{p_0}\cohm{0}{\bb CP^1}{\mcal O(2m)}$ is injective. It remains to find the image of the map $p_0$, in cohomology, above. 
		
		Since $l$ is even, we can write $l=2j$, for $0\leq j\leq m$. Define $v_j$ by the equation \eqref{eq:solution2} and notice that $\{v_0,\cdots,v_m\}$ is a global frame for $E_m$. Thus, for $\alpha\in\cohm{0}{\bb CP^1}{E_m}$, we can write $\alpha=\sum_{j=0}^m\alpha_jv_j$ and we have: 
		\begin{equation}\label{eq:image}
			p_0(\alpha)=\sum_{j=0}^m\alpha_j\xi^{2j}\in\cohm{0}{\bb CP^1}{\mcal O(2m)}.
		\end{equation}
		Using our notation, this means that the image of $p_0$ is $\Gamma_m$. This implies that sections of the form \eqref{eq:image} can be extended uniquely to $E_m$ and hence to the $m\ts{th}$ formal neighbourhood.
		
		An extension of sections given by \eqref{eq:image} on the $(m+1)\ts{th}$ neighbourhood is given by the pull-back to $E_{m+1}(-2)$ in the exact sequence:
		\begin{align}
			0\to E_{m}(-4)\to E_{m+1}(-2)\xrightarrow{p_0} \mcal O(2m)\to0.
		\end{align}
		However, in this case $\cohm{0}{\bb CP^1}{E_{m+1}(-2)}=0$ and no extension exists.
		
		As a concern of notation, if $s$ is a section in $\Gamma_m$, its formal extension in $L^{(m)}(2m)$ will be denoted by $\overline{s}$.
	\end{proof}
	
	Before we proceed we shall state the following lemma, whose proof is similar to the proof of lemma (5.2) in \cite{CM}:
	\begin{lem}\label{lem:h1unique}
		Every element $c\in\cohm{1}{S}{O(2k-2)}$ can be written uniquely in the form:
		\begin{align*}
			c=\sum_{i=[k+1/2]}^{2k-2}\eta^i\pi^*c_i,
		\end{align*}
		where $c_i\in\cohm{1}{\bb CP^1}{O(2k-2-4i)}$.
	\end{lem}
	
	\begin{proof}[Proof of proposition \eqref{prop:fibreat0}]
	Let us start with the exact sequence:
		\begin{align*}
			0\to\mcal O(-2m-4)\to L^{(m+1)}(2m)\to L^{(m)}(2m)\to0.
		\end{align*}
		Form its exact sequence in cohomology we have a map $$\delta:\cohm{0}{\bb CP^1}{L^{(m)}(2m)}\to\cohm{1}{\bb CP^1}{O(-2m-4)}.$$ Since $\cohm{0}{\bb CP^1}{L^{(m+1)}(2m)}=0$  and $\cohm{0}{\bb CP^1}{L^{(m)}(2m)}=\Gamma_m$ from lemma \eqref{lem:extension}, we can define an injective map $$h:\Gamma_m\to\cohm{1}{\bb CP^1}{\mcal O(-2m-4)},$$ defined by $hs=\delta\overline s$. 
		
		Let $s\in\Gamma_{k-1}$ and take the extension of $\pi^*s\in\cohm{0}{S}{\mcal O(2k-2)}$ to the order $k-1$, as in lemma \eqref{lem:extension}, and consider it to be a section of $L^z(2k-2)$ over $\bb C\times S$. The obstruction to extending to the order $k$ is the element  $$ c=\eta^k\pi^*hs\in\cohm{1}{S}{\mcal O(2k-2)}.$$ Now, since $S$ satisfies $\eta^k+a_1\eta^{k-1} +\cdots+a_0=0,$ we must have $$ c=-\sum a_i\eta^{k-i}\pi^*hs.$$
		Then we can write the above as:
		$$ c=-\sum \eta^{k-j}\pi^*h_j,$$
		where $h_j\in\cohm{1}{\bb CP^1}{\mcal O(4j-2k-2)}$ and also each $h_j$ must be in the image of $h$. Therefore, for each $j$ we can find a unique section $s_i\in\Gamma_{k-1-2j}$ such that $ \eta^{k-j}\pi^*h_j$ is the obstruction to extend $\pi^*s_j\in\cohm{0}{S}{\mcal O(2k-2-4j)}$ to the order $(k-2j)$ as a section of $L^z(2k-2-4j)$. This is the obstruction to extending $z^{2j}\eta^j\pi^*s_j$ from the order $(k-1)$ to the order $k$. 
		Therefore, if $\overline{s}$ denotes a formal extension, we have that
		\begin{align*}
			s^1=\overline{s}-z^2\eta\overline{s}_1-z^4\eta^2\overline{s}_2-\cdots-z^2l\eta^l\overline{s}_l
		\end{align*}
		extends to the order $k$ in $z$. Now, we can consider an extension of $s^1$ whose obstruction is $c'\in\cohm{1}{S}{\mcal O(2k-2)}$. We can proceed as above we shall add modifications of order $z^3$. Then, every coefficient of $z^n$ requires a finite number of modifications and we have a power series in $z$. Now we can use a result in \cite{Har} (proposition II 9.6) to prove that a convergent extension exists.
		We have then proved that $\pi^*(\Gamma_{k-1})\subset V_0$. Since both vector spaces have dimension $k$, we have proved the proposition.
	\end{proof}
	
	\begin{rmk}
		An important remark here is that since $\Gamma_m$ is not a natural irreducible representation of $SL(2,\bb C)$, the maps in cohomology in the proof above are interpreted only as maps of abelian groups and not as maps between irreducible representations. Therefore, the fibre of $V$ at $z=0$ does not have a natural $SL(2,\bb C)$ representation structure. This is an important difference between our case and the $\bb R^3$ case.
	\end{rmk}


\subsection{The behaviour of the matrix $A$ at $0$}
	
After having established the fibre of $V$ at $0$, we can move toward the description of the behaviour of the matrix $A(z,\xi)$ at $0$. Namely, we shall prove that $A(z,\xi)$ has a pole at $0$ and $2$ and describe the respective residues. 

As before, we shall work with $\tilde M$ instead of $M$. Remember that in corollary \eqref{cor:ztozn} we defined a map $\rho:\tilde V \to V$, which is an isomorphism away from $z=0$. Also, remember that $X=P_*(\mcal L(2k+2))$ and $\tilde X=\tilde P_*(\mcal L(2k+2))$. Then, we can state the following lemma:

\begin{lem}
	The diagram:	
	\begin{align*}
		\begin{CD}
		\tilde V @>\tilde F>> \tilde X\\
		@VV\rho V @VV\rho V\\
		V @>F>> X
		\end{CD}
	\end{align*}
	
	
	is commutative if either $F=z\eta$ and $\tilde F=\eta$ or $F=\tilde F= A(\xi,\eta)$.
\end{lem}

The proof of this lemma is direct from lemma \eqref{lem:ztozn} and corollary \eqref{cor:ztozn}. We now have the following:

\begin{cor}\label{cor:AtoB}
	Define $B(\xi,\eta)=z A(\xi\eta)$, then $(\eta-A(\xi,\eta))V=0$ if and only if $(\eta-B(\xi,\eta))\tilde V=0$.
\end{cor}

We shall now study the behaviour of $B$ at $z=0$ and use this corollary to deduce the corresponding behaviour of $A$. To start with this, we shall use Beauville's construction of $B$.

We start with the commutative diagram:
	\begin{align*}
		\begin{CD}
		\tilde V_z @>restr_{z\text{,}q} >> \cohm{0}{zS\cap T_q}{L(2k-2)}\cong \bb C^{k-1}\\
		@VVB(\xi\text{,}\eta)V @VV\times\eta V\\
		\tilde V_z @>restr_{z\text{,}q} >> \cohm{0}{zS\cap T_q}{L(2k-2)}\cong \bb C^{k-1}
		\end{CD}
	\end{align*}
where $q\in\bb CP^1$, $T_q$ is the fibre of $\bb T$ over $q$ and $$restr_{z,q}: \cohm{0}{zS}{L(2k-2)}\to \cohm{0}{zS\cap T_q}{L(2k-2)}$$ is the natural restriction map. Moreover, as in the construction of $A$, the cohomologies in the diagram above can be interpreted as polynomials in $\eta$ of degree $k-1$.

Observe that $restr_{z,q}$ is an isomorphism for all $z\neq0$ and its limit $restr_{0,q}$ is also an isomorphism. Now, let $\tilde e_0,\cdots,\tilde e_{k-1}$ be a local frame for $\tilde V$, in a neighbourhood of $0$, such that $restr_{0,q}(\tilde e_j)=\eta^j$.

Then $B$ is well-defined and continuous at $z=0$ and, if $\xi_0$ correspond to the point $q\in\bb CP^1$, 
\begin{align*}
B(0,\xi_0)(\tilde e_j)=\tilde e_{k+1}.
\end{align*} 
Since $B=zA$, we must have that $A$ has simple poles at $z=0$ and the next objective will be the description of the residues of $A$ at $0$.


Now we shall use the alternative description of $A$ given in \eqref{prop:kernelm} to find the residues of $A$. This means we shall investigate the behaviour of the kernel $K_z$ of the product map $$ m:\cohm{0}{S}{\mcal O(4)}\otimes\cohm{0}{S}{L^z(2k-2)}\to\cohm{0}{S}{L^z(2k+2)}$$ as $z\to 0$.
We start by noticing that, under the embedding $\bb T\subset \bb CP^5$, finding $K_0$ is equivalent to finding which sections of $\cohm{0}{S}{\Omega^1_{\bb CP^5}(2k+2)}$ extend to $\cohm{0}{S}{L^z\Omega^1_{\bb CP^5}(2k+2)}$. Since $dim K_z=k$ for $z\in(0,2)$, we should have a $k$-dimensional subspace  $K_0$ that extends. Next, we shall describe $K_0$.


Let $\{1,\xi^2,\cdots,\xi^{2k}\}$ be a basis for $\Gamma_{k-1}$ and define the linear operators $B_0,B_1$ and $B_2$ in $\Gamma_{k-1}$ by the the matrices:

		\begin{align}\label{eq:X_0}
				X_0=\left(\begin{array}{ccccc}
						0 & 0 & 0 & \ldots & 0 \\
						-(k-1) & 0 & 0 & \ldots & 0 \\
						0 & -(k-2) & 0 & \ldots & 0 \\
						\vdots & \vdots & \vdots & \vdots & \vdots \\
						0 & \cdots & \cdots & -1 & 0 
					\end{array} \right),
		\end{align}
		
		\begin{align}\label{eq:X_2}
				X_2=\left(\begin{array}{ccccc}
						k & 0 & 0 & \ldots & 0 \\
						0 & (k-2) & 0 & \ldots & 0 \\
						0 & 0 & (k-4) & \ldots & 0 \\
						\vdots & \vdots & \vdots & \vdots & \vdots \\
						0 & \cdots & \cdots & 0 & -k 
					\end{array} \right)\,\,\text{and}
		\end{align}
		
		\begin{align}\label{eq:X_4}
				X_4=\left(\begin{array}{ccccc}
						0 & 1& 0 & \ldots & 0 \\
						0 & 0 & 2 & \ldots & 0 \\
						\vdots & \vdots & \vdots & \vdots & \vdots \\
						0 & \cdots &  \cdots & 0 & (k-1)\\
						0 & \cdots &  \cdots & 0 & 0 
					\end{array} \right)
		\end{align}

we can now state the following result:

\begin{prop}\label{prop:bextension}
Every element $s\in K_0$ can be written uniquely in the form 
	 $$ s=\pi^*(1\otimes X_0\hat s+\xi^2\otimes X_2\hat s+\xi^4\otimes X_4\hat s),$$
	where $\hat s\in\Gamma_{k-1}$.
\end{prop}
\begin{proof}
	The idea of the proof of this proposition is to work on the $(k-1)\ts{th}$ order neighbourhood first. We shall find a basis for the fibre of the bundle $V$ at $0$ in the formal neighbourhood in the language of the lemma \eqref{lem:extension}, this is to say, we have to solve \eqref{eq:matrix1}.

	In what follows, we shall find $$P'^j=p'^j_0+p'^j_1(z\eta)+\cdots+p'^j_{(k-1)}(z\eta)^{(k-1)}$$ and $$P^j=p^j_0+p^j_1(z\eta)+\cdots+p^j_{(k-1)}(z\eta)^{(k-1)}$$ satisfying \eqref{eq:matrix1}, for $0\leq j\leq(k-1)$. 

	In what follows we shall use $m=k-1$ for simplicity.

	Fix $j$ and and define on the open set $U_0$:
\begin{align}\label{eq:plinha}
	p'^j_l= \left\{ \begin{array}{l}
         (-1)^l\dfrac{(m-l)!}{(m-l-j)!}\dfrac{(m-j)!}{m!}\dfrac{1}{l!}\xi^{-2(m-j-l)}\,\,\text{for}\,\,0\leq l\leq(m-j),\\
        0\,\,\text{otherwise}.
        \end{array} \right.
\end{align}
 
	And on $U_1$:
\begin{align}\label{eq:psemlinha}
	p^j_l= \left\{ \begin{array}{l}
         \dfrac{(m-l)!}{(j-l)!}\dfrac{j!}{m!}\dfrac{1}{l!}\xi^{2(j-l)}\,\,\text{for}\,\,0\leq l\leq j,\\
        0\,\,\text{otherwise}.
        \end{array} \right.
\end{align}

	We now need to check this data satisfies \eqref{eq:matrix1}. Let $$\beta_b=\left(\dfrac{1}{b!}\xi^{(2m-2b)},\dfrac{1}{(b-1)!}\xi^{(2m-2b-2)},\cdots,\xi^{(2m-4b)},0,\cdots,0 \right)$$ be the $b\ts{th}$ line of the matrix \eqref{eq:matrix1}. We need to prove that $$\beta_b\cdot P'^j =p^j_b.$$
	
	\begin{align*}
		\beta_b\cdot P^j&=\sum^{\text{min}\{b,m-j\}}_{l=0}(-1)^l\dfrac{(m-l)!(m-j)!}{(b-l)!(m-j-l)!m!l!}\xi^{(2j-2b)}\\
					 &=\dfrac{(m-b)!}{m!}\left[\sum^{\text{min}\{b,m-j\}}_{l=0}(-1)^l{m-l \choose b-l}{m-j \choose l}\right]\xi^{(2j-2b)}\\
					 &=\dfrac{(m-b)!}{m!}{j \choose b}\xi^{(2j-2b)}\\
					 &=\dfrac{(m-b)!j!}{(j-b)!b!m!}\xi^{(2j-2b)}=p^j_b.
	\end{align*}
Where we used the identity:
\[
 \sum^{\text{min}\{b,m-j\}}_{l=0}(-1)^l{m-j \choose l}{m-l \choose b-l}={j\choose b}.
\]

Then, we have that $P^j$ gives a basis for $V_0$ in the $(k-1)\ts{th}$ neighbourhood.

Now we shall describe the kernel of the multiplication map
\begin{align*}
	m:\cohm{0}{F^{(k-1)}}{\mcal O(4)}\otimes\cohm{0}{F^{(k-1)}}{L(2k-2)}\to\cohm{0}{F^{(k-1)}}{L(2k+2)}.
\end{align*}

First, notice that  we have $\cohm{0}{F^{(k-1)}}{\mcal O(4)}\cong\cohm{0}{\bb T}{\mcal O(4)}=Span_\bb C\{1,\xi,\xi^2,\xi^3,\xi^4,\eta\}$. Now, a direct computation shows that the kernel of $m$ is generated by the elements of the form:
\[
\omega^z_j=[z\eta\otimes P^j]-(m-j)[1\otimes P^{(j+1)}]+(m-2j)[\xi^2\otimes P^j]+j[\xi^4\otimes P^{(j-1)}],
\]
for $0\leq j\leq(k-1)$. In other words, this says that we can find sections $t_0,s_0,s_2,s_4\in\Gamma_{k-1}$ such that
$$z\eta\overline t_0 +\overline s_0+\overline s_2\xi^2+\overline s_4\xi^4=0\,\, \text{mod}z^k,$$
where $\overline t_0$ and $\overline s_j$ represent the canonical extensions of $t_0$ and $s_j$ respectively. Moreover, we have proved above that we can actually take $t_0=s$ and $s_j=X_j(s)$, for $j=0,2,4$, for $s\in\Gamma_{k-1}$.
	
	Now, the canonical extension is of order $(k-1)$ and we proceed as in the proof of proposition \eqref{prop:fibreat0} to extend to higher orders and produce a formal extension. We can use again a result in \cite{Har} (proposition II 9.6) to prove that the obstruction to extend to higher orders are removable and, therefore we can produce an actual extension.
\end{proof}

\begin{rmk}
	It is important to highlight how we found the solutions \eqref{eq:plinha} and $\eqref{eq:psemlinha}$ to \eqref{eq:matrix1}. We solved \eqref{eq:matrix1} explicitly, from $k=2$ up to $k=6$, using the constraints \eqref{eq:solution1} and then we obtained a pattern for the solution for general $k$. In the proof written here, we just used this general form of the solution and proved it actually solves \eqref{eq:matrix1}.
\end{rmk}

	We can now use this to prove our main result:

\begin{thm}\label{thm:main}
	Let $S$ be a spectral curve in $\bb T$ satisfying the conditions in definition \eqref{spectral}with charge $k$. Then the corresponding matrices $A_i$, that satisfy the Nahm's equations \ref{eq:cNahm's}, also satisfy the following boundary conditions:
	\begin{enumerate}
		\item $A_1$ and $A_3$ are analytic on the whole interval $[0,2]$;
		\item $A_0$, $A_4$ and $A_2$ have simple poles at $0$ and $2$, but are otherwise analytic;
		\item The residues of $A_0$, $A_4$ and $A_2$ at $z=0$ and $z=2$ define an irreducible $k$-dimensional representation of $\mathfrak{sl}(2,\bb C)$.
	\end{enumerate}
\end{thm}
\begin{proof}
	Remember that from corollary \eqref{cor:AtoB} the endomorphisms $B_j$, defined by $B_j=zA_j$, are analytic on the whole interval $[0,2]$. Moreover, the proposition above tells us that $B_1$ and $B_3$ vanish at $z=0$ and:
	$$\lim_{z\to0}zB_j(s)= X_j(s),$$
	for $j=0,2,4$.
	
	This means the endomorphisms $A_1$ and $A_3$ are analytic on the whole interval $[0,2]$ and the endomorphisms $A_0,A_2$ and $A_4$ have simple poles at $0$ whose residues are given by $X_0,X_2$ and $X_4$ respectively. We now shall extend this to the matrices that appear on the Nahm's equations.
	We have that the covariant derivative in $V$ is defined by:
	\begin{align*}
		\nabla_zs=\frac{\partial f_0}{\partial z}-\left(\dfrac{1}{2}A_2s+\xi A_3s+\xi^2A_4s\right).
	\end{align*}
	From the above and the definition of $X_j$( equations \eqref{eq:X_0}, \eqref{eq:X_2} and $\eqref{eq:X_4}$)  we have:
	\begin{align*}
		\dfrac{1}{2}A_2+\xi A_3+\xi^2A_4=\dfrac{(k-1)}{2z}\times\bb I+D,
	\end{align*}
	where $D$ is analytic in the whole interval $[0,2]$ and $\bb I$ is the $k\times k$ identity matrix.
	Since the residue of the connection is a scalar, we can use the same argument in \cite{CM} page 179 to conclude that the matrices $A_j$ have the same residue as the corresponding endomorphisms. Thus, $A_0, A_4$ and $A_2$ define an irreducible representation of $\mathfrak{sl}(2,\bb C)$. Moreover, the condition that $L^2$ is trivial on $S$ says that the behaviour of the residues at $z=2$ is the same as at $z=0$.
	\end{proof}
	
\begin{rmk}
	If we have $A_0=T_1+iT_2$, $A_1=T_3+iT_4$, $A_2=2iT_5$, $A_3=T_3-iT_4$ and $A_4=-T_1+iT_2$, then we have that $T_3$ and $T_4$ are analytic on the whole interval $[0,2]$ and the residues of $T_1$, $T_2$ and $T_5$ at $0$ define an irreducible representation of $\mathfrak{sl}(2,\bb R)$.
\end{rmk}

	


\subsection{From Nahm's equations to spectral curves}

The goal of this subsection is to prove the following theorem:

\begin{prop}\label{prop:converse}
	Let $A_j:(0,2)\to \mathfrak gl(k)$, $j=0,1,2,3,4$ satisfy the Nahm's equations and such that:
	\begin{enumerate}
	\item $A_1$ and $A_3$ are analytic on the whole interval $[0,2]$.
	\item $A_0, A_4$ and $A_2$ are analytic on $(0,2)$ and have simple poles at $z=0$ and $z=2$ with residues $a_0, a_4$ and $a_2$ defining the irreducible representation of $\mathfrak sl(2,\bb C)$ of rank $k$. 
	\end{enumerate}	 
	 Then, the curve $S$ defined by $det(\eta-A)=0$, where $A=A_0+\xi A_1+\xi^2 A_2+\xi^3 A_3+\xi^4 A_4$, satisfies:
	\begin{enumerate}[i)]
	\item $S$ is compact,
	\item $L^2$ is trivial on $S$,
	\item $\cohm{0}{S}{L^z(2k-3)}=0$ for $z\in(0,2)$.
	\end{enumerate}
\end{prop}
\begin{proof}
	For part i) notice that, $det(\eta-A)$ can be written, in local coordinates, as the polynomial in the equation $\ref{eq:spectralcurve}$. Thus, $S$ is compact.
	
	We now start to invert the procedure we used to construct the $A(\xi,t)$. Namely, using Beauville's theorem, we obtain a flow of line bundles $K_t$ on $S$. More explicitly, given the matrix $A(\xi,t)$ we have that 
	$$K_t=coker(\eta-A(\xi,t)),$$
	where
	$$(\eta-A(\xi,t)):\mcal O(-4)^{\oplus k}\to\mcal O^{\oplus k}.$$ 
	However, it is easier to consider the dual approach. This means we are going to find the dual flow:
	$$K_t^*=ker(\eta-A(\xi,t))^t,$$
	where
	$$(\eta-A(\xi,t))^t:\mcal O^{\oplus k}\to\mcal O(4)^{\oplus k}.$$
	
	First we shall prove that $K^*_t=K^*_{t_0}\otimes L^{t-t_0}$. We start with a section $s$ of $K_{t_0}^*$ and it can be represented by $u$ in the open set $\{\xi\neq\infty\}$ and by $v$ on $\{\xi\neq 0\}$. Moreover, let $g(t_0)$ be the transition function of $K^*_{t_0}$ such that $u=g(t_0)v$. Observe that on $\{\xi\neq\infty\}$ we must have:
	$$ (\eta-A(\xi,t_0))^tu=0$$
	and on $\{\xi\neq 0\}$:
	$$ (1/\xi^4)(\eta-A(\xi,t_0))^tv=0.$$
	
	Let $A_+=\frac{1}{2}A_2+A_3\xi+A_4\xi^2$ and we shall vary $t$. To begin with, we impose that $u$ satisfies:
	$$ \der{u}{t}=A_+^tu.$$
	We can use Nahm's equation to prove that 
	$$\der{}{t}(\eta-A)^tu=A_+(\eta-A)^tu .$$
	Now, the initial condition for this differential equation is given by $(\eta-A)^tu=0$. Thus, we have $(\eta-A)^tu=0$ for all $t$.
	
	On the other open set we can impose
	$$\der{v}{t}=-\left(A/\xi^2-A_+\right)^tv $$
	and prove that
	$$ (1/\xi^4)(\eta-A)^tv=0$$
	for all $t$.
	Now we have:
	\begin{align*}
		A_+^t=\der{u}{t}=\der{gv}{t}=\der{g}{t}v-g\left(A/\xi^2-A_+\right)^tv.
	\end{align*}
	This implies that
	$$ \frac{\eta}{\xi^2}u=\der{g}{t}g^{-1}u.$$
	The solution of this equation can be written in terms of $g(t_0)$ as $g(\eta,\xi,t)=e^{t\eta/\xi^2}\cdot g(t_0)$. Therefore, we proved that $K^*_t=K^*_{t_0}\otimes L^{t-t_0}$.
	
	We now move towards the description of $K_0$ and we shall use the boundary behaviour of the matrices $A_i$ to prove that $K_0\cong\mcal O(2k-2)$.
	
	Near $t=0$ we can write for $t>0$ $A(\xi,t)=\dfrac{\alpha(\xi,t)}{t}$, with $\alpha(\xi,t)$ analytic near $t=0$. Also, denote $a(\xi,t)=\alpha(\xi,t)^t$. Write $\alpha(0,\xi)=a(\xi)=a_0+a_1\xi+a_2\xi^2+a_3\xi^3+a_4\xi^4$. From our hypothesis, $a_j$ satisfy the conditions in theorem \eqref{thm:main}. This means that $a_1=a_3=0$ and $a_0,a_2$ and $a_4$ define an irreducible representation of $\mathfrak{sl}(2,\bb C)$. 
	
	Let $\Gamma_{k-1}$ be the subspace of $\bb C^{2k-1}$ consisting of polynomials of the form $p(\xi)=\sum_{i=0}^{2k-2}\xi^{2i}$. The matrices $a_j$ act on $\Gamma_{k-1}$ by multiplication. Moreover, we can choose a basis $e_i$ for $\Gamma_{k-1}$ such that $ker [a(\xi)]=(\xi^{2k-2},\cdots,\xi^{2j},\cdots,1)$. Notice that in this basis, $a_0(e_i)=e_{i+1}$.
	
	We shall next compute a section of $ker(\eta-A)^t$, first observe that 
	$$(\eta-A)^t(\eta-A)^t_{adj}=det(\eta-A)\times\bb I,$$ 
	where $adj$ is the formal adjoint and $\bb I$ is the identity matrix. This means that on $S$, $Im(\eta-A)^t_{adj}\subset K^*_t$. However, since $(\eta-A)$ is regular $(\eta-A)^t_{adj}$ has rank one and the inclusion becomes an equality. 
	
	Now, we shall compute a section of $(\eta-A)^t_{adj}$. Observe that at $\xi=0$, we have that the image of $(\eta-A)^t_{adj}$ has a finite limit, because of the choice of basis above. In the general case, $Im(\eta-A)^t_{adj}\subset K^*_t$ will consist of a polynomial of degree $2k-2$ in $\xi$ and therefore, $K_0\cong\mcal O(2k-2)$. This means that $K_t=L^t(2k-2)$. 
	
	Notice that, since the behaviour of the matrices $T_j$ at $t=2$ are the same as at $t=0$, we also have $K_2\cong\mcal O(2k-2)$. This implies that $L^2=0$. Lastly, from Beauville's theorem we must have $K_t(-1)\in J^{g-1}_S$, this is to say, $\cohm{0}{S}{L^t(2k-3)}$ for all $t\in(0,2)$.
\end{proof}



\bibliographystyle{amsplain}
\bibliography{refs}
\nocite{Spec,HAH,HM}

\end{document}